\documentclass[11pt]{amsart}
\usepackage{amsmath,amsthm,amssymb,bm}
\usepackage{xcolor}
\usepackage{indentfirst}
\usepackage{graphicx}
\usepackage{tikz}
\usetikzlibrary{decorations.pathreplacing, decorations.pathmorphing, decorations.shapes}
\usetikzlibrary{arrows,calc}
\usepackage{mathrsfs,esint}
\usepackage[nocompress]{cite}
\usepackage{comment}
\usepackage{enumerate}
\usepackage{hyperref}
\hypersetup{
    colorlinks = true,
    linkcolor = red,
    anchorcolor = green!50!black,
    citecolor = green!50!black,
    filecolor = green!50!black,
    urlcolor = green!50!black
}
\usepackage{cleveref}
\crefname{thm}{Theorem}{Theorems}
\Crefname{thm}{Theorem}{Theorems}

\usepackage[top=1in, bottom=1in, left=1.25in, right=1.25in]{geometry}
\usepackage{setspace}
\usepackage{float}

\usepackage{subcaption}
\usepackage{pgfplots}
\pgfplotsset{compat=newest, compat/show suggested version=false}
\usepackage{tikz}
\usetikzlibrary{intersections}

\usepackage{diagbox}
\usepackage{algorithm, algorithmicx, algpseudocode}
\numberwithin{algorithm}{section}
\usepackage{tabularx}

\usepackage[titletoc,toc,title]{appendix}

\newtheorem{thm}{Theorem}[section]
\newtheorem{theorem}[thm]{Theorem}

\newtheorem{lemma}[thm]{Lemma}
\newtheorem{remark}[thm]{\textit{Remark}}
\newtheorem{definition}[thm]{Definition}

\newtheorem{example}[thm]{Example}

\newcommand{\titleShort}{Robust Policy Evaluation with L\'evy Dynamics}
\newcommand{\titleLong}{A Robust Model-Based Approach for Continuous-Time Policy Evaluation with Unknown L\'evy Process Dynamics}

\newcommand{\drift}{b}
\newcommand{\diffusionOrdinary}{D_{\textnormal{o}}}
\newcommand{\diffusionFractional}{D_{\textnormal{f}}}

\usepackage{orcidlink}

\title[\titleShort]{\titleLong}

\thanks{}

\author{Qihao~Ye~\orcidlink{0000-0002-7369-757X}}
\address{Department of Mathematics, University of California, San Diego, CA 92093, United States} 
\email{q8ye@ucsd.edu}

\author{Xiaochuan~Tian~\orcidlink{0000-0002-4539-6702}}
\address{Department of Mathematics, University of California, San Diego, CA 92093, United States} 
\email{xctian@ucsd.edu}

\author{Yuhua~Zhu~\orcidlink{0009-0000-7197-218X}}
\address{Department of Statistics and Data Science, University of California, Los Angeles, CA 90095, United States}
\email{yuhuazhu@ucla.edu}

\date{}

\numberwithin{equation}{section}

\begin{document}

\begin{abstract}
    This paper develops a model-based framework for continuous-time policy evaluation (CTPE) in reinforcement learning, incorporating both Brownian and L\'evy noise to model stochastic dynamics influenced by rare and extreme events.
Our approach formulates the policy evaluation problem as solving a partial integro-differential equation (PIDE) for the value function with unknown coefficients.
A key challenge in this setting is accurately recovering the unknown coefficients in the stochastic dynamics, particularly when driven by L\'evy processes with heavy tail effects.
To address this, we propose a robust numerical approach that effectively handles both unbiased and censored trajectory datasets.
This method combines maximum likelihood estimation with an iterative tail correction mechanism, improving the stability and accuracy of coefficient recovery.
Additionally, we establish a theoretical bound for the policy evaluation error based on coefficient recovery error.
Through numerical experiments, including a real-data BTC price experiment, we demonstrate the effectiveness and robustness of our method in recovering heavy-tailed L\'evy dynamics and verify the theoretical error analysis in policy evaluation.

\end{abstract}



\maketitle

\onehalfspacing


\section{Introduction}\label{sec:introduction}
Reinforcement learning (RL) has achieved remarkable success in artificial intelligence, with applications such as AlphaGo \cite{Silver2016}, strategic gameplay \cite{Mnih2015}, and fine-tuning large language models \cite{Ziegler2019}.
However, these successes are primarily in discrete-time sequential decision-making settings, where the state changes only after an action is taken.
In contrast, in most real-world decision-making problems, the state evolves continuously in time, regardless of whether actions are taken in continuous or discrete time.
Examples include dynamic treatment regimes in healthcare \cite{hua2022personalized,murphy2003optimal}, robotics \cite{karimi2023dynamic,Kober2013,siciliano1999robot}, autonomous driving \cite{sciarretta2004optimal}, and financial markets \cite{merton1975optimum,Moody2001}.

One common approach to handling continuous-time reinforcement learning (CTRL) is to discretize time and reformulate the problem as a discrete-time Markov decision process (MDP) \cite{doya2000reinforcement}.
This transformation allows standard RL algorithms to be applied directly within the classical RL framework.
However, as demonstrated in \cite{zhu2024phibe}, for policy evaluation, the discretization error can be significant when the reward function exhibits large oscillations.
Moreover, in RL, reward functions often need to oscillate significantly to effectively distinguish between rewards and penalties, which is essential for learning an optimal policy.
This requirement suggests that MDP may not always be an ideal framework for solving CTRL problems, a limitation that has also been observed empirically \cite{naik2024reward,sowerby2022designing}.
Fundamentally, the MDP framework is designed for discrete-time decision-making.
In other words, even when the continuous-time structure of a problem is known, MDPs lack a natural mechanism to fully leverage this information.
A complementary line of work treats RL directly through continuous-time stochastic control, including exploratory control formulations, martingale policy evaluation, policy-gradient and actor-critic methods, and continuous-time $q$-learning \cite{wang2020reinforcement,jia2022policya,jia2022policyb,jia2023q}.

This paper focuses on the continuous-time policy evaluation (CTPE) problem, which is one of the most fundamental problems in RL.
The continuous-time dynamics under a given policy are governed by the following stochastic differential equation (SDE):
\begin{equation}\label{eq:underlying_dynamics}
    d X_{t}
    = \drift(X_{t}) \, d t
    + \Sigma(X_{t}) \, d W_{t}
    + \sigma(X_{t}) \, d L_{t}^{\alpha},
\end{equation}
where $W_{t}$ is a standard Wiener process, and $L_{t}^\alpha$ is a symmetric $2 \alpha$-stable L\'evy process with $\alpha \in (0, 1)$.
In general, the original continuous-time dynamics depends on the action.
However, in the policy evaluation problem, we substitute the policy directly, resulting in continuous-time dynamics that depend only on the state. We assume that the drift term ${\drift}(x)$ and diffusion terms $\Sigma(x), \sigma(x)$ are unknown and independent of time.
Unlike traditional continuous-time models that often assume $\sigma \equiv 0$, our model incorporates both Brownian motion and non-Gaussian, heavy-tailed L\'evy processes.
Many real-world stochastic processes, such as financial returns, network traffic in communication systems, anomalous diffusion in physics, are better described by L\'evy processes rather than purely Gaussian models \cite{clauset2009power,huang2023model,humphries2010environmental,MeKl00,samorodnitsky1996stable,tankov2003financial,zaslavsky2002chaos}.
This generalized framework enhances versatility in capturing complex stochastic behavior in real-world systems.
The goal of CTPE is to estimate the value function
\begin{equation}\label{eq:definition_of_the_value_function} 
    V(x)
    = \mathbb{E} \left [ \left . \int_{0}^{\infty} e^{-\beta t} r(X_{t}) dt \right | X_{0} = x \right ]
\end{equation}
using only trajectory data generated by the underlying dynamics \eqref{eq:underlying_dynamics}.
Here, $\beta$ is a given discounted coefficient, and $r$ is a given reward function.
The key difference between CTPE and policy evaluation in the MDP framework is that CTPE aims to estimate an integral over continuous time, whereas in discrete-time decision-making problems, the objective is to estimate a cumulative sum over discrete time.

One effective approach to leverage the underlying continuous-time structure is to interpret the value function in \eqref{eq:definition_of_the_value_function} as the solution to the following partial integro-differential equation (PIDE), which is derived in \Cref{lemma:policy_evaluation}:
\begin{equation*}
    \beta V(x)
    = r(x)
    + \drift(x) \cdot \nabla V (x)
    + \diffusionOrdinary(x) : \nabla^{2} V (x)
    - \diffusionFractional(x) (- \Delta)^{\alpha} V(x).
\end{equation*}
Here, the coefficients $\drift(x)$, $\diffusionOrdinary(x)$, and $\diffusionFractional(x)$ related to the underlying dynamics are unknown and must be inferred from trajectory data.
Compared to the MDP framework, addressing the CTRL problem within a model-based PDE framework offers several advantages.
Furthermore, while the MDP framework inherently suffers from an $\mathcal{O}(\Delta t)$ discretization error due to time discretization, the error in the PDE framework depends instead on the accuracy of the estimated coefficients.
In certain cases, this allows one to mitigate the $\mathcal{O}(\Delta t)$ error from discretized observations and achieve a more accurate estimate of the value function \cite{gobet2004nonparametric}.
Another key advantage of the PDE formulation is its interpretability.
When trajectory data are abundant and prior knowledge of the underlying dynamics is available, the model-based approach can effectively incorporate this information to refine the estimated dynamics, leading to more robust and reliable decision-making \cite{amos2021model,janner2019trust}.
Finally, this model-based policy evaluation framework can be easily extended to continuous-time control settings.
By integrating it with existing optimal control solvers, it provides a natural approach for addressing the CTRL problem.
These advantages make the PDE-based approach a compelling alternative to the traditional MDP-based framework, particularly in scenarios where the continuous-time structure plays a crucial role in decision-making and learning.

Our approach to policy evaluation consists of two main steps: first, recovering the coefficient functions of the stochastic dynamics from trajectory data, and second, solving the PIDE based on the estimated model.
To estimate the coefficients, we use maximum likelihood estimation combined with an efficient fractional Fokker--Planck equation solver.
For the second step, we establish a theoretical bound for the policy evaluation error in \Cref{thm:error_bound}, which quantifies the impact of approximation errors in the recovered coefficients.
As our results suggest, the accuracy of policy evaluation is directly dependent on how well the coefficient functions are estimated.
A key challenge in this process is learning stochastic dynamics driven by L\'evy noise.
This problem has been considered in several previous works \cite{guo2024weak,li2021data,li2022extracting,yang2023neural}.
However, existing approaches are either limited to relatively large values of $\alpha\in (1 / 2, 1)$ or assume a constant coefficient for the L\'evy noise term.
In this work, we aim to provide a comprehensive study of the performance of our approach in recovering the stochastic dynamics \eqref{eq:underlying_dynamics} with variable coefficients by examining the impact of different datasets and $\alpha$ values on the numerical approach. 
The difficulties are twofold, both stemming from the presence of a $2 \alpha$-stable L\'evy process.
First, coefficient recovery is often unstable when dealing with strongly heavy-tailed data, particularly when $\alpha$ is small, as illustrated in the left panel of \Cref{fig:comparison}.
This instability arises due to large jumps in the data, making coefficient estimation highly sensitive.
Second, tail data may be missing due to measurement limitations in real-world datasets or computational constraints imposed by sampling algorithms such as Markov Chain Monte Carlo (MCMC).
In such cases, while coefficient recovery becomes more stable, it is also less accurate, as shown in the right panel of \Cref{fig:comparison}.
To address these issues, we propose a novel approach specifically designed for data with insufficient tail information (or unbiased datasets where tails are intentionally removed).
Our method incorporates an iterative tail correction mechanism that mitigates instability while improving recovery accuracy.
Numerical results, presented in the middle panel of \Cref{fig:comparison}, demonstrate the effectiveness and robustness of our approach.

\begin{figure}[htp]
    \centering
    \includegraphics[width=\linewidth]{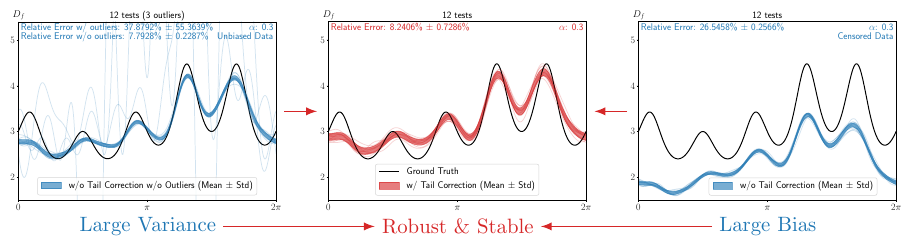}
    \caption
    {
        Comparison of different approaches in estimating $\diffusionFractional(x)$ across 12 independent tests for each case.
        Left panel: results using unbiased trajectory data, showing instability in recovery with a large number of outliers. 
        Right panel: results using censored trajectory data with filtered tails, leading to more stable recovery but introducing significant bias.
        Middle panel: results incorporating the tail correction technique, which improves the accuracy and robustness of coefficient recovery.
        Further details on the numerical tests can be found in \Cref{sec:numerical_experiments}.
    }
    \label{fig:comparison}
\end{figure}

The key contributions of this paper are as follows.
\begin{enumerate}
    \item \textit{The L\'evy process-based model.}
    While many previous RL models assume stochastic dynamics with only Brownian noise, we investigate RL problems in which the underlying stochastic dynamics are influenced by both Brownian noise and L\'evy noise.
    This framework is well-suited for modeling real-world scenarios where rare and extreme events occur in markets or systems.
    \item \textit{Accurate and robust recovery of L\'evy dynamics.}
    The proposed numerical approach effectively addresses the challenges of recovering L\'evy dynamics from strongly heavy-tailed data.
    This method is especially crucial when the coefficient $\alpha \in (0, 1)$ is small, corresponding to distributions with more severe tails.
    \item \textit{A theoretical bound for policy evaluation error.}
    Our PIDE-based approach to policy evaluation comes with a theoretical guarantee of accuracy.
    The policy evaluation error is ultimately determined by the recovery error of the stochastic dynamics and the numerical error introduced in solving the PIDE.
\end{enumerate}

\noindent {\bf Related work.}
There are two primary types of inverse problems in reinforcement learning.
The first concerns inferring the reward function given known dynamics, as studied in works such as \cite{ding2022mean,levine2012continuous,ren2024policy}.
The second, which is the focus of this paper, involves recovering unknown dynamics when the reward function is known.
Our work is also connected to the recent continuous-time RL literature based on stochastic control and exploratory regularization \cite{wang2020reinforcement,Guo22,tang2022exploratory,han2023choquet}.
A growing body of literature has addressed the identification of continuous stochastic dynamics from observational data, primarily focusing on systems driven by standard Brownian motion, without incorporating jump or L\'evy processes.
When trajectory data are available, relevant approaches include learning the coefficients of SDEs \cite{pmlr-v235-chen24n,darcy2023one,dietrich2023learning,gu2023stationary,zhu2024dyngma,zhu2023learning}, as well as flow-map-based models that directly learn the underlying dynamics \cite{chen2024learning,liu2024training,xu2024modeling}.
Other works utilize aggregate data to infer the dynamics of the system \cite{chen2021solving,ma2021learning,yang2022generative}. 
In contrast, relatively few studies have addressed the recovery of continuous-time dynamics that incorporate jump processes.
Studies on identifying dynamics driven by counting processes or compound Poisson processes include \cite{gao2022data,jia2019neural,xia2024efficient}.
The continuous-time RL framework has recently been extended from diffusions to jump-diffusions \cite{gao2024reinforcement}, while related work on identifying $2 \alpha$-stable L\'evy processes appears in \cite{guo2024weak, li2021data, li2022extracting, yang2023neural}.
These studies typically consider either large fixed values of $\alpha$ (e.g., $\alpha \in (1/2, 1)$), corresponding to less severe heavy tails, or assume a constant fractional diffusion coefficient $\sigma(x) \equiv \textnormal{const}$.
In this paper, we propose a new algorithm that extends to \textit{state-dependent} diffusion coefficients $\sigma(x)$ and is applicable across a wide range of $\alpha$.
Moreover, the method remains robust under censored trajectory data, addressing key limitations of the existing literature.

The remainder of the paper is organized as follows. \Cref{sec:framework} presents the notation and problem setup, introducing the CTPE problem along with the necessary SDE and PDE tools.
In \Cref{sec:numerical_methods}, we describe our numerical approaches for recovering stochastic dynamics and evaluating policies.
Specifically, we introduce a novel tail correction technique designed to enhance the accuracy and robustness of L\'evy process recovery.
A theoretical error bound is also given for policy evaluation.
Numerical results, including a real-data BTC experiment, are provided in \Cref{sec:numerical_experiments}, followed by conclusions in \Cref{sec:conclusion}.


\section{Problem Setting}\label{sec:framework}
Consider the following continuous-time policy evaluation problem, where the value function $V^{\pi}(x) \in \mathbb{R}$ is the expected discounted cumulative reward starting from ${x}$ under policy $\pi$:
\begin{equation*}
    V^{\pi}(x)
    = \mathbb{E} \left [ \left . \int_{0}^{\infty} e^{- \beta t} r(X_{t}, a_{t}) \, d t \right | X_{0} = x, a_{t} \sim \pi \right ].
\end{equation*}
Here $\beta > 0$ is a discount coefficient, $r(x, a) \in \mathbb{R}$ is a reward function, and the state $X_{t} \in \mathbb{S} = \mathbb{R}^{d}$ at time $t$ is a Markov stochastic process.
The state variable $X_{t}$ satisfies the SDE 
\begin{equation*}
    d X_{t}
    = \drift(X_{t}, a_{t}) \, d t
    + \Sigma(X_{t}, a_{t}) \, d W_{t}
    + \sigma(X_{t}, a_{t}) \, d L_{t}^{\alpha}
\end{equation*}
where $\drift(x, a) \in \mathbb{R}^{d}$, $\Sigma(x, a) \in \mathbb{R}^{d \times M}$, and $\sigma(x, a) \in \mathbb{R}$.
$W_{t}$ is an $M$-dimensional standard Wiener process, and ${L}_{t}^{\alpha}$ is a symmetric $2 \alpha$-stable L\'evy process with $\alpha \in (0, 1)$.
For a given policy $\pi$, the problem simplifies to \Cref{eq:definition_of_the_value_function} following the dynamics described in \Cref{eq:underlying_dynamics}.
The refinement of the policy update mechanism based on this policy evaluation problem is left for future research.

Define $\diffusionOrdinary(x) := \frac{1}{2} \Sigma(x) \Sigma^{T}(x)$ and $\diffusionFractional(x) := |\sigma(x)|^{2 \alpha}$.
In this work, we assume $\diffusionOrdinary(x)$ is uniformly positive definite and $\diffusionFractional(x) \geq \lambda_{\textnormal{f}} > 0$.
The probability density function $p(x, t)$ of $X_{t}$ starting from $x_{0} \in \mathbb{R}^{d}$ at initial time $t = 0$ is described by the following fractional Fokker--Planck equation (FFPE) \cite{gao2016fokker,nolan2020univariate,schertzer2001fractional}:
\begin{equation}\label{eq:FFPE_condition_on_x_0}
    \left \{
    \begin{aligned}
        &
        \begin{aligned}
            \frac{\partial}{\partial t} p(x, t )
            = &- \nabla \cdot \left [ \drift(x) p(x, t) \right ]
            + \nabla^{2}:\left [ \diffusionOrdinary(x) p(x, t ) \right ] 
            - (- \Delta)^{\alpha} \left [ \diffusionFractional(x) p(x, t) \right ]
        \end{aligned}
        \\
        &p(x, 0)
        = \delta_{x_{0}}(x)
    \end{aligned}
    \right.,
\end{equation}
where $\nabla^{2}$ is the Hessian operator and the double-dot ($:$) represents a tensor contraction by summing over pairs of matching indices.

Furthermore, the value function $V({x})$ defined above solves the second-order partial integro-differential equation in \Cref{lemma:policy_evaluation}.
\begin{lemma}\label{lemma:policy_evaluation}
Given the probability density function of the stochastic process $X_t$ governed by the fractional Fokker--Planck equation \eqref{eq:FFPE_condition_on_x_0}, the corresponding value function defined in \eqref{eq:definition_of_the_value_function} satisfies
    \begin{equation}
    \label{eq:PDE_value_function}
        \beta V(x)
        = r(x)
        + \drift(x) \cdot \nabla V(x)
        + \diffusionOrdinary(x) : \nabla^{2} V(x)
        - \diffusionFractional(x) (- \Delta)^{\alpha} V(x).
    \end{equation}
\end{lemma}
\Cref{eq:PDE_value_function} can be interpreted in the {\it viscosity sense}, with the precise definition of viscosity solutions given in \Cref{defn:viscositysolu}.
By the regularity result proved in \Cref{lem:regularity}, viscosity solutions coincide with classical solutions with additional appropriate assumptions.
Finally, the proof of the above lemma follows directly from Theorem 9.1, Chapter III in \cite{fleming2006controlled} by observing that the operator $\drift \cdot \nabla + \diffusionOrdinary : \nabla^{2} - \diffusionFractional (- \Delta)^{\alpha}$ is the infinitesimal generator of the Markov process described by \Cref{eq:underlying_dynamics}.
We note that although the theorem in \cite{fleming2006controlled} is stated for classical solutions, we know that it is also true for viscosity solutions given the regularity estimates in \Cref{lem:regularity}.

The objective of this paper is to solve \eqref{eq:PDE_value_function} with unknown coefficient functions $\drift(x)$, $\diffusionOrdinary(x)$, and $\diffusionFractional(x)$.
This involves two steps: first, recovering the coefficient functions from the observed data, and second, solving \eqref{eq:PDE_value_function} using the recovered approximate coefficients.
In practice, we only have access to observed trajectory data, denoted by $\{ x_{j \Delta t}^{(i)} \}_{i = 1, j = 0}^{i = I, j = J}$, where $i$ indexes different trajectories and $j$ represents discrete time steps.
The trajectory data start from the initial states $x_{0}^{(i)}$ for $i = 1, \ldots, I$, and the subsequent states $x_{(j + 1) \Delta t}^{(i)}$ ($j = 0, \ldots, J - 1$) evolve according to the probability density function $p(x, \Delta t \mid x_{j \Delta t}^{(i)})$, governed by the first equation in \eqref{eq:FFPE_condition_on_x_0} with initial condition at $x_{j \Delta t}^{(i)}$.
In our numerical experiments, we consider the following two different types of trajectory data.
\begin{enumerate}
    \item
    \textbf{Unbiased trajectory data}:
    We assume that the trajectory data are unbiased samples from the stochastic dynamics \eqref{eq:underlying_dynamics}.
    In practice, to generate such data, we employ the Euler-Maruyama approximation of \eqref{eq:underlying_dynamics} with a relatively small time step ($\sim \Delta t / 10$).
    At each time step, we apply the sampling algorithms from \cite{nolan2020univariate}.
    \item
    \textbf{Censored trajectory data}:
    In contrast to the unbiased data, the censored trajectory data systematically omit certain parts of the dynamics, particularly the tails or large jumps.
    This missing information may arise due to observational constraints, filtering mechanisms, or inherent biases in data collection.
    As a result, the censored data set underrepresents rare but significant transitions.
    To generate such data in practice, we employ two approaches:
    (1) using the same Euler-Maruyama approximation with small time steps and filtering out a percentage of large trajectory segments exceeding a predefined threshold, and
    (2) generating data with an MCMC sampler constrained to a limited range.
\end{enumerate}


\section{Numerical Methods}\label{sec:numerical_methods}
The policy evaluation process based on the given trajectory data has two main stages.
In the first stage, the coefficients $\drift(x), \diffusionOrdinary(x), \diffusionFractional(x)$ are recovered using maximum likelihood estimation.
In this step, the probability density functions are computed using a fast numerical solver for the FFPE \eqref{eq:FFPE_condition_on_x_0}.
A tail correction technique is introduced to improve the accuracy and robustness of coefficient recovery, especially for severely heavy-tailed data.
See \Cref{sec:coefficient_recovery_method,sec:TC_technique} for details.
In the second stage, the value function is solved from \Cref{eq:PDE_value_function} using the recovered coefficients.
A theoretical error bound is provided in \Cref{sec:policy_evaluation_method} given the coefficient recovery error.
Comprehensive numerical experiments for coefficient recovery and policy evaluation are presented in \Cref{sec:underlying_dynamics_recovery_experiment} and \Cref{sec:policy_evaluation_experiment}, respectively.

To better illustrate our approach, we focus on one-dimensional setups in this section and in the numerical examples of the next section.
It is important to note that in the coefficient recovery step, numerical solvers are required to compute the probability density functions of the underlying stochastic dynamics. 
In \Cref{appendix:numerical_method}, we present a tailored adaptation of the fast solver from \cite{ye2026fast} for efficiently solving the following FFPE with constant coefficients:
\begin{equation}\label{eq:FFPE_constant}
    \left \{
    \begin{aligned}
        \frac{\partial}{\partial t} p(x, t)
        &= - \frac{\partial}{\partial x} \left [ \drift(x_{0}) p(x, t) \right ] + \frac{\partial^{2}}{\partial x^{2}} \left [ \diffusionOrdinary(x_{0}) p(x, t) \right ]
        - (- \Delta)^{\alpha} \left [ \diffusionFractional(x_{0}) p(x, t) \right ]\\
        p(x, 0)
        &= \delta_{x_{0}}(x)
    \end{aligned}
    \right ..
\end{equation}
Note that \Cref{eq:FFPE_constant} provides a short-time approximation of \Cref{eq:FFPE_condition_on_x_0}. The accuracy of this approximation improves as the time frame $\Delta t$ decreases.
We acknowledge that the approximation error to the FFPE with variable coefficients is a major source of error in recovering the stochastic dynamics.
Therefore, developing a solver for the FFPE with variable coefficients remains a key objective for future research.

\subsection{Coefficient Recovery}\label{sec:coefficient_recovery_method}
In this subsection, we present numerical methods for estimating the unknown coefficients from the discrete-time trajectory data, as described in \Cref{sec:framework}.
In the main synthetic experiments, we assume that the fractional exponent $\alpha \in (0, 1)$ is known.
The real-data experiment in \Cref{sec:real_data_experiment} includes a constant-$\alpha$ fit; a systematic treatment of unknown or state-dependent $\alpha$ is left for future work.
We approximate the coefficient functions from a finite-dimensional space spanned by Fourier basis:
\begin{equation*}
    \mathcal{S}_{N}
    = \text{span} \{ 1, \cos(k x), \sin(k x) : k = 1, \ldots, N \}.
\end{equation*}
Let $K = 2 N + 1$, and let $\{ \phi_{k} \}_{k = 1}^{K}$ denote the basis functions in $\mathcal{S}_{N}$.
For clarity of presentation, we assume that the coefficient functions belong to $\mathcal{S}_{N}$ and are expanded as
\begin{equation}\label{eq:basis_function_expansion}
    \drift(x; \theta)
    = \sum_{k = 1}^{K} \theta_{1, k} \phi_{k}(x), \quad
    \diffusionOrdinary(x; \theta)
    = \sum_{k = 1}^{K} \theta_{2, k} \phi_{k}(x), \quad
    \diffusionFractional(x; \theta)
    = \sum_{k = 1}^{K} \theta_{3, k} \phi_{k}(x).
\end{equation}
We note that using different basis functions to represent the coefficient functions is also viable.
Alternative representations, such as neural network-based approaches, are also possible and are deferred to future discussion. To facilitate our discussion, we define 
$\Theta(x; \theta) := \left ( \drift(x; \theta), \diffusionOrdinary(x; \theta), \diffusionFractional(x; \theta) \right )$ and $\theta = \{ \theta_{l, k} \}_{l = 1, k = 1}^{l = 3, k = K}$.

We aim to maximize the log-likelihood function given by
\begin{equation}\label{eq:log_likelihood_function_approximation}
    \ell(\theta)
    \approx \sum_{i = 1}^{I} \sum_{j = 0}^{J - 1} \ln p \left ( x_{(j + 1) \Delta t}^{(i)}, \Delta t; x_{j \Delta t}^{(i)}, \alpha, \Theta \left ( x_{j \Delta t}^{(i)}; \theta \right ) \right ),
\end{equation}
where $p \left ( x, \Delta t; x_{j \Delta t}^{(i)}, \alpha, \Theta \left ( x_{j \Delta t}^{(i)}; \theta \right ) \right )$ denotes the solution to the FFPE with initial condition $\rho_{0}(x) = \delta \left ( x - x_{j \Delta t}^{(i)} \right )$, fractional exponent $\alpha$ and coefficients $\Theta \left ( x_{j \Delta t}^{(i)}; \theta \right )$.
Maximizing this function through gradient-based methods facilitates the estimation of $\theta$.

The gradient
\begin{equation}
    \nabla_{\theta} \ell =
    \left( 
        \frac{\partial \ell}{\partial \theta_{1, 1}}, \ldots, \frac{\partial \ell}{\partial \theta_{1, K}},
        \frac{\partial \ell}{\partial \theta_{2, 1}}, \ldots, \frac{\partial \ell}{\partial \theta_{2, K}},
        \frac{\partial \ell}{\partial \theta_{3, 1}}, \ldots, \frac{\partial \ell}{\partial \theta_{3, K}}
    \right)^{T}
\end{equation}  
can be computed either through direct computation or the finite difference approximation, with the latter being more general but less stable and computationally more expensive.
A detailed discussion on the implementation of the maximum likelihood estimation is provided in \Cref{appendix:maximum_likelihood_estimation}.
In particular, \Cref{alg:FFPE_direct_gradient} provides an algorithm for direct gradient computation and \Cref{alg:FFPE_finite_difference_gradient} outlines the finite difference approximation.

\subsection{The Tail Correction Technique}\label{sec:TC_technique}
The tail correction technique presented in this subsection specifically addresses the {\it censored trajectory data} described in \Cref{sec:framework}.
This serves two key purposes.
First, when using censored data, the recovered fractional diffusion coefficient consistently underestimates the true value, especially for small $\alpha$, as illustrated in \Cref{fig:comparison}.
The proposed tail correction technique corrects this bias.
Second, when working with unbiased data, the presence of heavy-tailed trajectories leads to instability in coefficient recovery, as also shown in \Cref{fig:comparison}.
Applying the tail correction technique to unbiased data with filtered tails enhances robustness in recovery.
Further details on numerical experiments are provided in \Cref{sec:numerical_experiments}.

More precisely, we use a cutting threshold $\texttt{CT} > 0$ to construct a tail sub-data pool $P_{\textnormal{tail}}$ from the full observation data pool $P_{\textnormal{main}}$:
\begin{equation}
\label{eq:tail_pool}
    P_{\textnormal{tail}}
    = \left \{ \left ( x_{\textnormal{current}}, x_{\textnormal{next}} \right ) \in P_{\textnormal{main}} \bigg | \left | x_{\textnormal{next}} - x_{\textnormal{current}} - \mu \right | > \texttt{CT} \right \},
\end{equation}
where the scalar $\mu$ is defined as the median of all consecutive differences in $P_{\textnormal{main}}$.
The median is preferred over the mean in the definition of $P_{\textnormal{tail}}$ because it leads to more stable and reliable recovery.
There is a practical method that can be used to determine \texttt{CT}, as described in \Cref{appendix:CT_selection}.
Here the full data pool is
\begin{equation}
\label{eq:main_pool}
    P_{\textnormal{main}}
    = \left \{ \left ( x_{j \Delta t}^{(i)}, x_{(j + 1) \Delta t}^{(i)} \right ) \right \}_{i = 1, j = 0}^{i = I, j = J - 1}.
\end{equation}
The set $P_{\textnormal{tail}}$ consists of transition pairs $\left ( x_{\textnormal{current}}, x_{\textnormal{next}} \right )$ with large jumps.
The tail ratio in the data is given by $\texttt{R}_{\textnormal{sample}} = |P_{\textnormal{tail}}| / |P_{\textnormal{main}}|$, where $| \cdot |$ here denotes the cardinality of a set.
This empirical ratio differs from the true tail ratio.
Given a probability density function $p$ with mean $\mu$, the true tail ratio is defined by $\texttt{R}_{\textnormal{actual}} = \int_{|x - \mu| > \texttt{CT}} p(x) \, d x$.
For the censored trajectory data, we usually have $\texttt{R}_{\textnormal{sample}} < \texttt{R}_{\textnormal{actual}}$. 
To correct the bias, we introduce a Tail Correction Factor (\texttt{TCF}) to place additional emphasis on $P_{\textnormal{tail}}$, which is determined from the following relation:
\begin{equation}\label{eq:correction_equation}
    \underbrace{(1 - \texttt{TCF}) \times \texttt{R}_{\textnormal{sample}}}_{\textnormal{normal sampling}} + \underbrace{\texttt{TCF} \times 1}_{\textnormal{tail sampling}}
    = \texttt{R}_{\textnormal{actual}}.
\end{equation}
This yields $\texttt{TCF} = (\texttt{R}_{\textnormal{actual}} - \texttt{R}_{\textnormal{sample}}) / (1 - \texttt{R}_{\textnormal{sample}})$.
The key idea is to modify the log-likelihood function $\ell(\theta)$ in \Cref{eq:log_likelihood_function_approximation} by introducing \texttt{TCF}-based sampling: at each step of stochastic gradient descent, data are drawn from $P_{\textnormal{tail}}$ with probability \texttt{TCF} and from $P_{\textnormal{main}}$ with probability $1 - \texttt{TCF}$. 
The influence of \texttt{TCF} values on the objective landscape is illustrated in \Cref{fig:TCF_explanation_landscape}.
In all the contour plots, the $x$-axis represents $\diffusionOrdinary$ and the $y$-axis represents $\diffusionFractional$, with known coefficient $\drift = 5$ and fractional exponent $\alpha = 0.3$.
These plots illustrate the influence of varying \texttt{TCF} on the location of the minimizer of the modified negative log-likelihood function with normalized output.

\begin{figure}[htp]
    \centering
    \includegraphics[width=\linewidth]{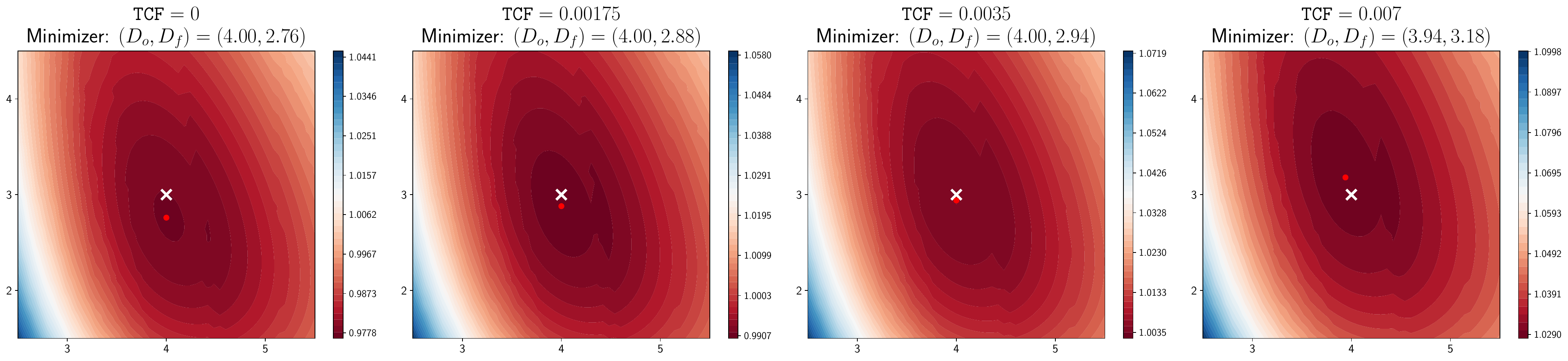}
    \caption
    {
        Contour plots showing the objective landscape for varying \texttt{TCF} values, with $\diffusionOrdinary$ on the $x$-axis and $\diffusionFractional$ on the $y$-axis.
        The white ``$\times$'' denotes the ground truth $(4, 3)$, while the red dot represents the minimizer of the objective for each respective \texttt{TCF} value.
        The observed shifts in the minimizer across different \texttt{TCF} values highlight the critical role of \texttt{TCF} selection in estimation accuracy. 
        This corresponds to a censored trajectory dataset obtained by an MCMC sampler.
    }
    \label{fig:TCF_explanation_landscape}
\end{figure}

To accurately determine the appropriate \texttt{TCF} value, it is essential to have knowledge of $\texttt{R}_{\textnormal{actual}}$, which in turn depends on the actual probability density function $p$ that we ultimately aim to recover.
Motivated by bootstrapping, we propose an adaptive method for recovering the coefficients and \texttt{TCF} together using the tail correction technique.
Specifically, we begin by estimating the coefficients under the assumption that \texttt{TCF} is zero.
Next, we update \texttt{TCF} using $\texttt{R}_{\theta}$ (substituting $\texttt{R}_{\textnormal{actual}}$ in \Cref{eq:correction_equation}), where $\texttt{R}_{\theta}$ is computed from the probability density function $p$ based on the previously estimated $\theta$.
This process is repeated iteratively until convergence is achieved or a preset iteration limit is reached.
\Cref{alg:Adam_with_tail_correction_concise} presents a concise version of coefficient recovery using the tail correction technique, and \Cref{alg:Adam_with_tail_correction} provides a detailed implementation.
The effectiveness and implementation of the adaptive tail correction technique are demonstrated in \Cref{fig:adaptive_paths}.
A detailed theoretical analysis of the \texttt{TCF} is beyond the scope of this paper and will be addressed in future work.

\begin{figure}[htp]
    \centering
    \includegraphics[width=0.38\linewidth]{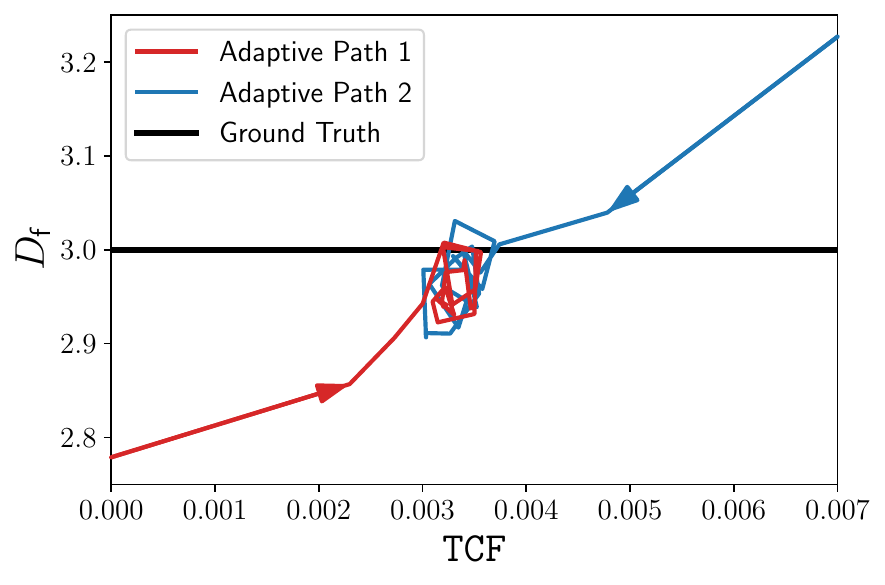}
    \hfill
    \includegraphics[width=0.21\linewidth]{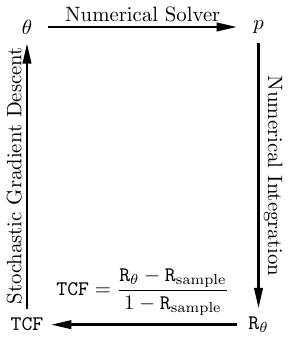}
    \hfill
    \includegraphics[width=0.38\linewidth]{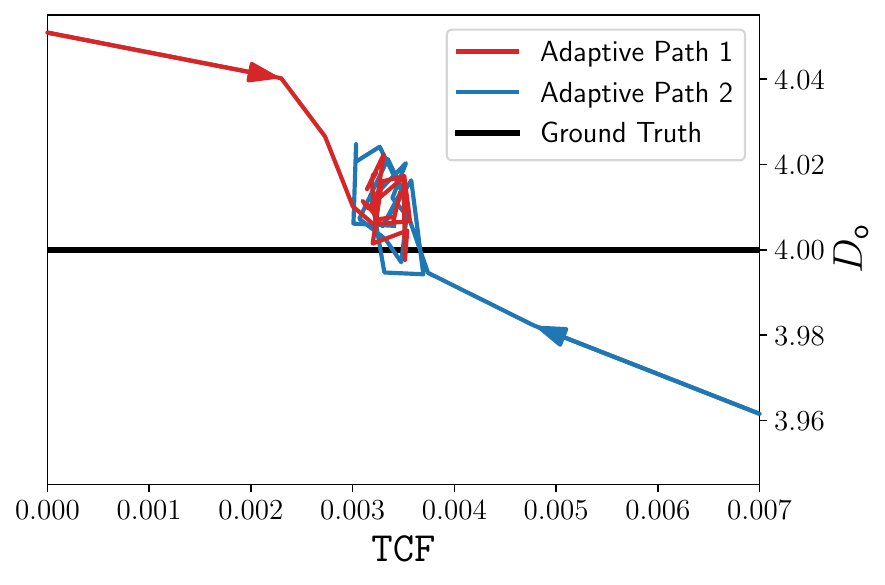}
    \caption
    {
        The left and right graphs illustrate how adaptively adjusting the \texttt{TCF} improves the accuracy of recovering the coefficients $\diffusionFractional$ and $\diffusionOrdinary$.
        The central diagram illustrates the process of adaptively applying the tail correction technique.
        This corresponds to a censored trajectory dataset obtained by an MCMC sampler.
        In practice, we update the \texttt{TCF} in each step to enhance the overall efficiency of the process, as outlined in \Cref{alg:Adam_with_tail_correction_concise}.
    }
    \label{fig:adaptive_paths}
\end{figure}

\begin{algorithm}[htbp]
    \caption{Robust Maximum Likelihood Estimation (Concise Version)}\label{alg:Adam_with_tail_correction_concise}
    \begin{algorithmic}[1]
        \Require All the trajectories $P$ and the cutting threshold \texttt{CT}
        \Ensure $\theta$ representing the estimated coefficients

        \State Randomly initialize $\theta$ to make sure $\diffusionOrdinary, \diffusionFractional \geq 0$;
        $\texttt{TCF} \gets 0$
        \State $P_{\textnormal{main}}, P_{\textnormal{tail}} \gets $ the main sample pool and the tail part sample pool based on \texttt{CT}
        \While{within the step limit \textbf{and} the moving average of $\theta$ is not converged}
            \State $\ell(\theta)$ is calculated by samples from $P_{\textnormal{tail}}$ with probability \texttt{TCF}, otherwise from $P_{\textnormal{main}}$
            \State $\boldsymbol{g} \gets$ the gradient of $- \ell(\theta)$
            \State Maintain the first and second moments $\boldsymbol{m}, \boldsymbol{v}$ of $\boldsymbol{g}$
            \State Update $\theta$ using $\boldsymbol{m}, \boldsymbol{v}$ and the learning rate, then update $\texttt{TCF}$ using $\theta$ and \texttt{CT}
        \EndWhile
            \Comment{Detailed version can be found at \Cref{alg:Adam_with_tail_correction}}
        \State \Return the moving average of $\theta$
    \end{algorithmic}
\end{algorithm}

\subsection{Policy Evaluation}\label{sec:policy_evaluation_method}
Policy evaluation involves solving the value function defined by the PDE given in \Cref{lemma:policy_evaluation} using the estimated coefficients
\begin{equation*}
    \widehat{\Theta}(x; \widehat{\theta})
    :=
    \left (
        \widehat{\drift}(x; \widehat{\theta}),
        \widehat{\diffusionOrdinary}(x; \widehat{\theta}),
        \widehat{\diffusionFractional}(x; \widehat{\theta})
    \right ).
\end{equation*}
For simplicity, in the remainder of the paper, we omit the dependence on $\theta$ and $\widehat{\theta}$ in $\Theta(x)$ and $\widehat{\Theta}(x)$, respectively, when the context is clear.
More precisely, we look for the approximate value function $\widehat{V}$ satisfying
\begin{equation}
\label{eq:PDE_value_function_approx}
    \beta \widehat{V}(x)
        = r(x)
        + \widehat{\drift}(x) \cdot \nabla \widehat{V}(x)
        + \widehat{\diffusionOrdinary}(x) : \nabla^{2} \widehat{V}(x)
        - \widehat{\diffusionFractional}(x) \left [ (- \Delta)^{\alpha} \widehat{V} \right ](x).
\end{equation}
To solve \Cref{eq:PDE_value_function_approx}, we apply the Fourier spectral method to determine $\widehat{V}$.

The following theorem, which holds true in general dimensions, provides the approximation error of $\widehat{V}$ to the true value function $V$, and the proof of it is given in \Cref{sec:proof_of_theorem}.

\begin{theorem}[Policy Evaluation Error]\label{thm:error_bound}
    Let $\gamma \in (0, 1)$ be a sufficiently small universal positive constant and $r(x), \Theta(x), \widehat{\Theta}(x)$ be $\gamma$-H\"older continuous and periodic functions defined on $\mathbb{R}^{d}$, with periodicity defined on the unit cell $Q = (0, 2 \pi]^{d}$.
    If $\| \Theta - \widehat{\Theta} \|_{C^{0, \gamma}(\mathbb{R}^{d})} \leq \epsilon$, then 
    \begin{equation*}
        \left \| V - \widehat{V} \right \|_{C^{2, \gamma}(\mathbb{R}^{d})}
        < C \epsilon
    \end{equation*}
    where $C > 0$ is independent of $\epsilon$, and $V, \widehat{V} \in C^{2, \gamma}(\mathbb{R}^{d})$ denotes the solution to \Cref{eq:PDE_value_function} and \Cref{eq:PDE_value_function_approx}, respectively.
\end{theorem}


\section{Numerical Experiments}\label{sec:numerical_experiments}
All experiments are conducted using MATLAB R2023a on a desktop equipped with an 11th Generation Intel\textsuperscript{\textregistered} Core\textsuperscript{\texttrademark} i7-11700F CPU and DDR4 2$\times$32GB 3600MHz memory.
Except for the real-data experiment in \Cref{sec:real_data_experiment}, we use synthetic data generated as described in \Cref{sec:framework}.
Data and code will be made available after the manuscript is accepted.

In \Cref{sec:underlying_dynamics_recovery_experiment}, we present numerical results for learning the underlying dynamics from both unbiased and censored trajectory data. 
\Cref{thm:error_bound} suggests that, within the proposed PDE framework, the error in evaluating the value function can be effectively managed by achieving accurate recovery of the coefficients.
This result is further validated through the numerical experiments presented in \Cref{sec:policy_evaluation_experiment}.

\subsection{Underlying Dynamics Recovery}\label{sec:underlying_dynamics_recovery_experiment}
When $\alpha$ is small, the trajectory data exhibits more frequent large jumps, making it more challenging to recover the underlying stochastic dynamics.
To address this, the tail correction technique described in \Cref{sec:TC_technique} proves useful.
In this subsection, we first present recovery results for unbiased trajectory data without applying the tail correction technique.
We then present the results for censored trajectory data where the technique is employed.
Such data may arise due to intrinsic bias in the data or from filtering out large jumps in unbiased trajectory data to enhance the stability of recovery.

In both \Cref{sec:unbiased_data,sec:censored_data}, we use trajectory data in which each trajectory consists of $41$ data points, corresponding to $40$ equal time intervals.
For each box plot, a total of 12 independent experiments are conducted.
Each experiment runs for 40,000 training steps, and a moving average with a window size of 20,000 steps is applied.
In all tests, we use the Adam optimizer \cite{kingma2014adam} to recover the coefficients, with gradients computed by \Cref{alg:FFPE_direct_gradient}.

\subsubsection{Unbiased Trajectory Data}\label{sec:unbiased_data}
\begin{example}\label{eg:unbiased_constant}
    The ground truth coefficients are $(\drift, \diffusionOrdinary, \diffusionFractional) = (5, 4, 3)$.
    We assume that we know the target coefficients are constant and set $K = 1$.
    The time difference between consecutive points is $\Delta t = 1 / 40$.
    No tail correction technique is applied in this example.
    A visualization of the results for both $\alpha = 0.3$ and $\alpha = 0.6$ is provided in \Cref{fig:unbiased_constant}.
\end{example}

As indicated in \Cref{fig:unbiased_constant}, as the number of trajectories increases, the error decreases, while the standard deviation of the estimates slightly increases.

\begin{figure}[htp]
    \centering
    \includegraphics[width=0.32\linewidth]{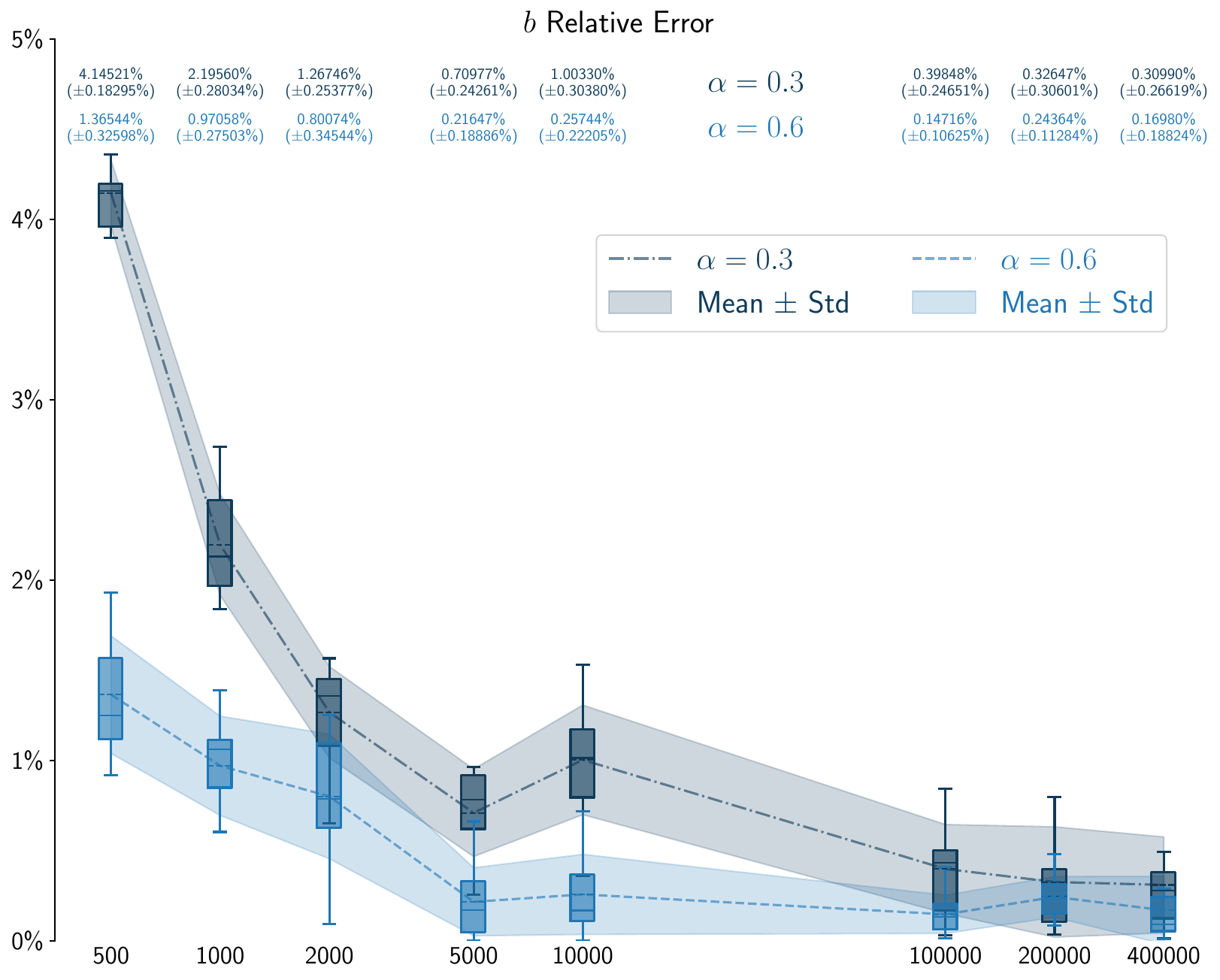}
    \hfill
    \includegraphics[width=0.32\linewidth]{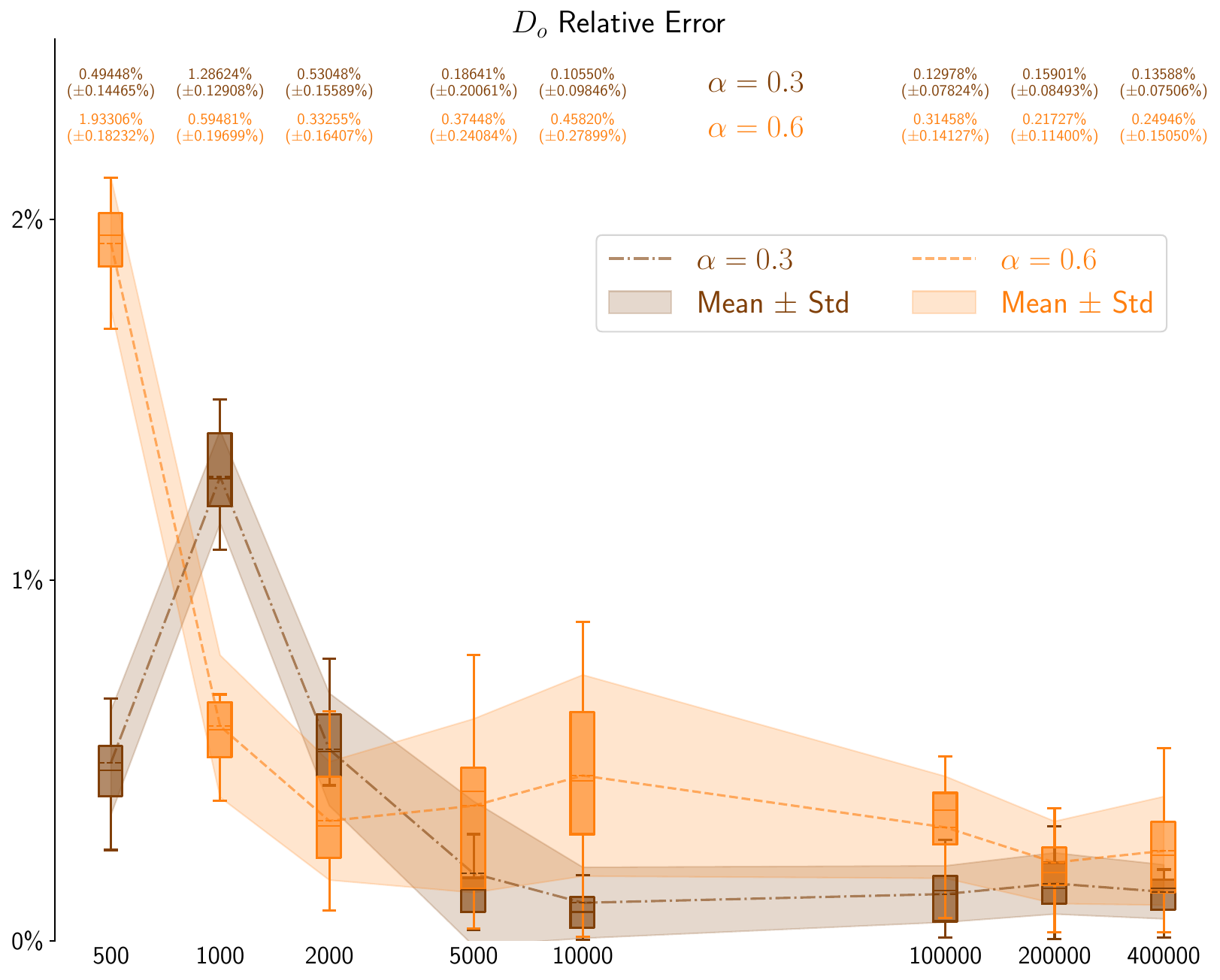}
    \hfill
    \includegraphics[width=0.32\linewidth]{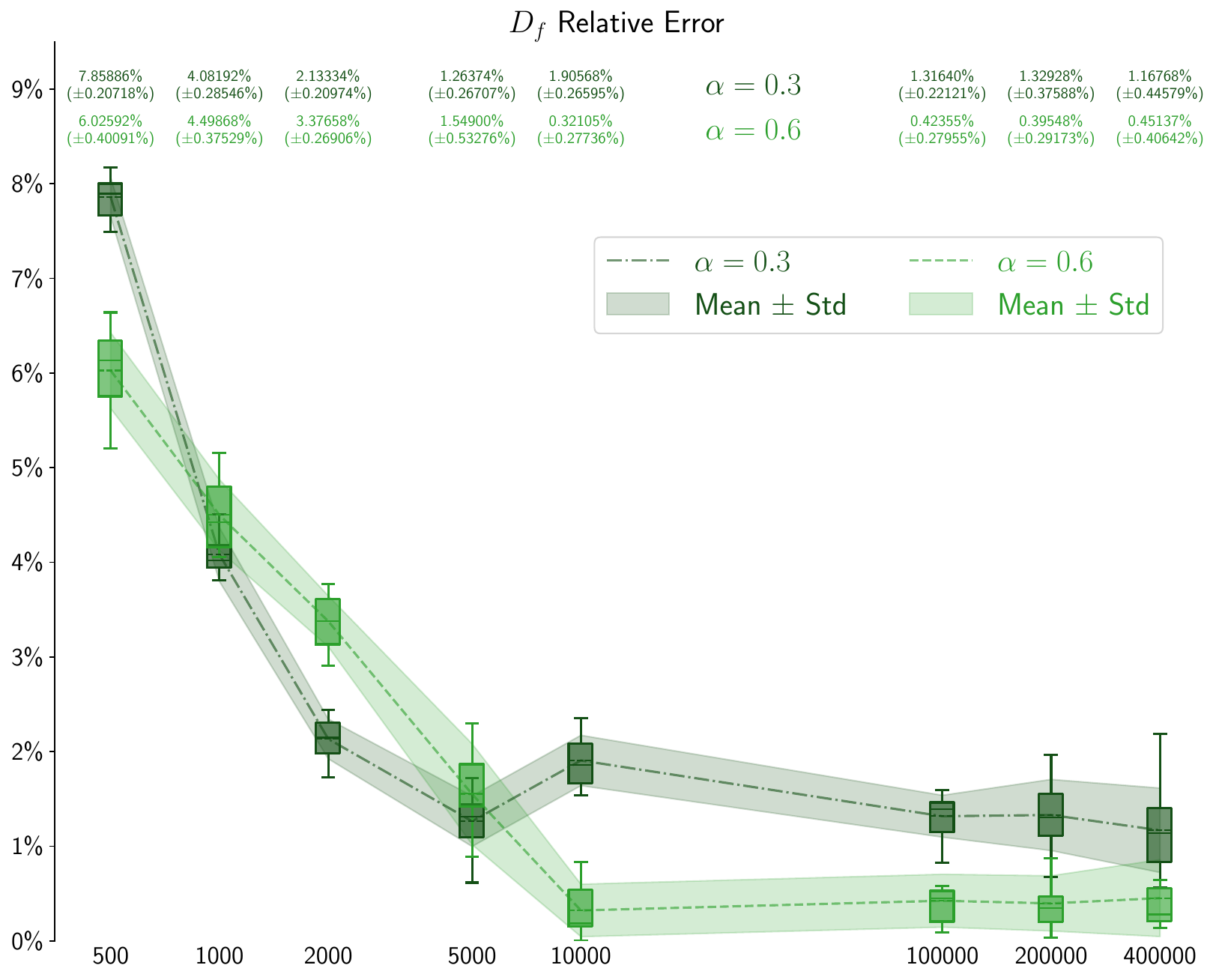}
    \caption
    {
        Relative errors of $\drift$ (left panel), $\diffusionOrdinary$ (middle panel), $\diffusionFractional$ (right panel) versus the number of trajectories for $\alpha = 0.3$ (deep color) and $\alpha = 0.6$ (light color).
        Related to \Cref{eg:unbiased_constant}.
    }
    \label{fig:unbiased_constant}
\end{figure}

\begin{example}\label{eg:unbiased_variable}
    The ground truth coefficients are
    $\drift(x) = 4 |(x \text{ mod } 2 \pi) - \pi| - 2 \pi$,
    $\diffusionOrdinary(x) = \exp(\sin(x + 1) + 1)$,
    $\diffusionFractional(x) = 2 + \exp(\sin(2 x) \cos(3 x))$ (depicted in the right panel of \Cref{fig:unbiased_variable}).
    We set $K = 21$.
    The time difference between consecutive points is $\Delta t = 1 / 400$.
    No tail correction technique is applied in this example.
    A visualization of the results for $\alpha = 0.3$ is presented in \Cref{fig:unbiased_variable}.
\end{example}

In the subsequent tests, we concentrate on the case $\alpha = 0.3$, as it poses a greater challenge for coefficient recovery, especially for variable coefficients.
In \Cref{eg:unbiased_variable}, we set $\Delta t = 1 / 400$ to reduce the impact of approximation error arising from the use of constant coefficients to approximate the variable-coefficient case at each time step.
The results corresponding to $\Delta t = 1 / 40$ with 400,000 trajectories will be presented in \Cref{eg:censored_half_tail_variable} as a baseline for comparison with the censored trajectory data, both with and without the application of the tail correction technique.
When $\alpha$ is small, the recovery results become unstable, with certain ``outliers'' deviating significantly from the ground truth.
This is linked to the strong influence of data with large jumps.
In the left panel of \Cref{fig:unbiased_variable}, we compute errors after removing these outliers.
The error decreases as the number of trajectories increases.
The right panel of \Cref{fig:unbiased_variable} shows the ground truth functions as solid lines.
The recovered functions, excluding outliers, are plotted with shaded bands representing the mean $\pm$ one standard deviation.
The outliers appear as opaque semi-transparent lines, showing their significant impact on the recovery error.

\begin{figure}[htp]
    \centering
    \includegraphics[width=.48\linewidth]{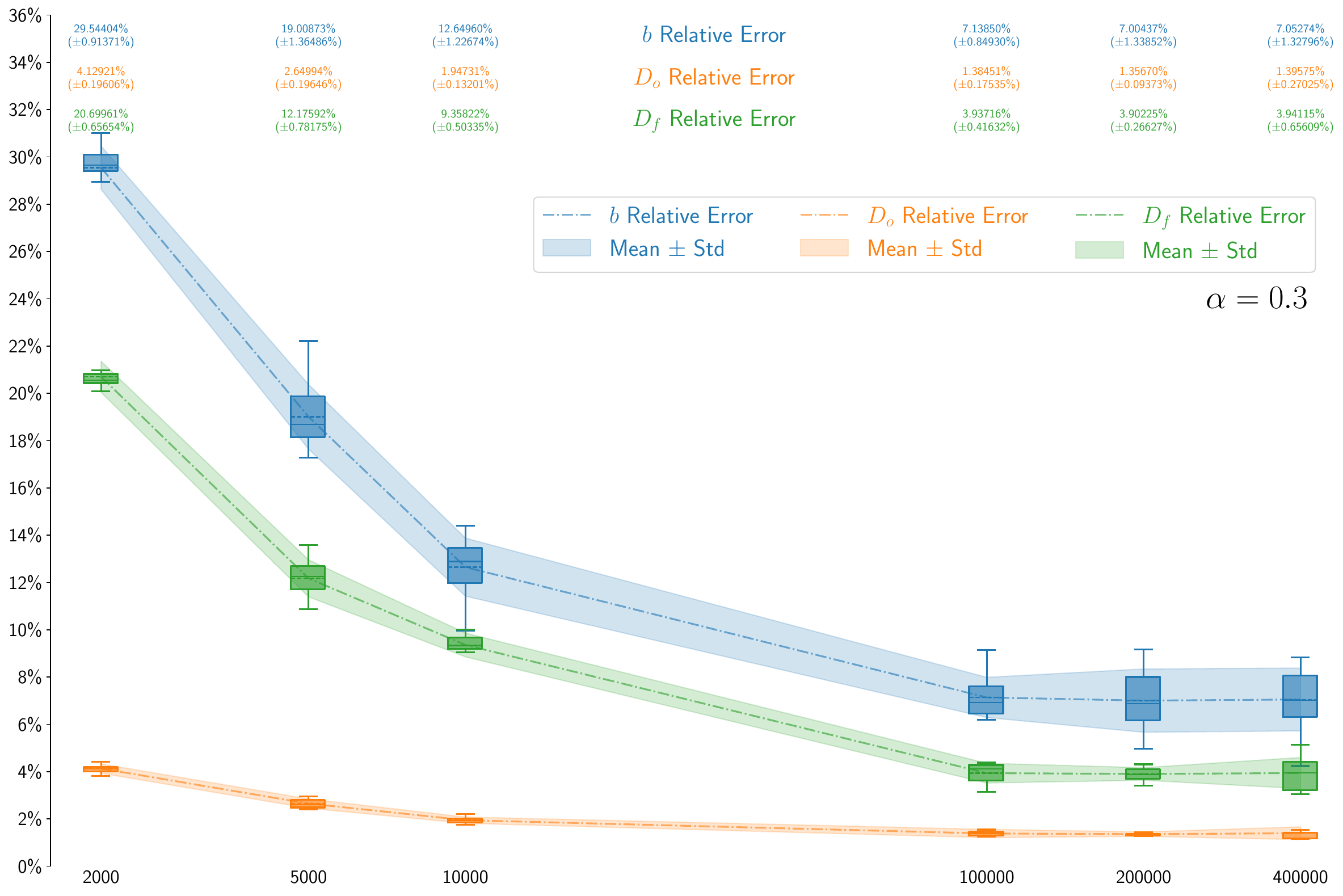}
    \hfill
    \includegraphics[width=.48\linewidth]{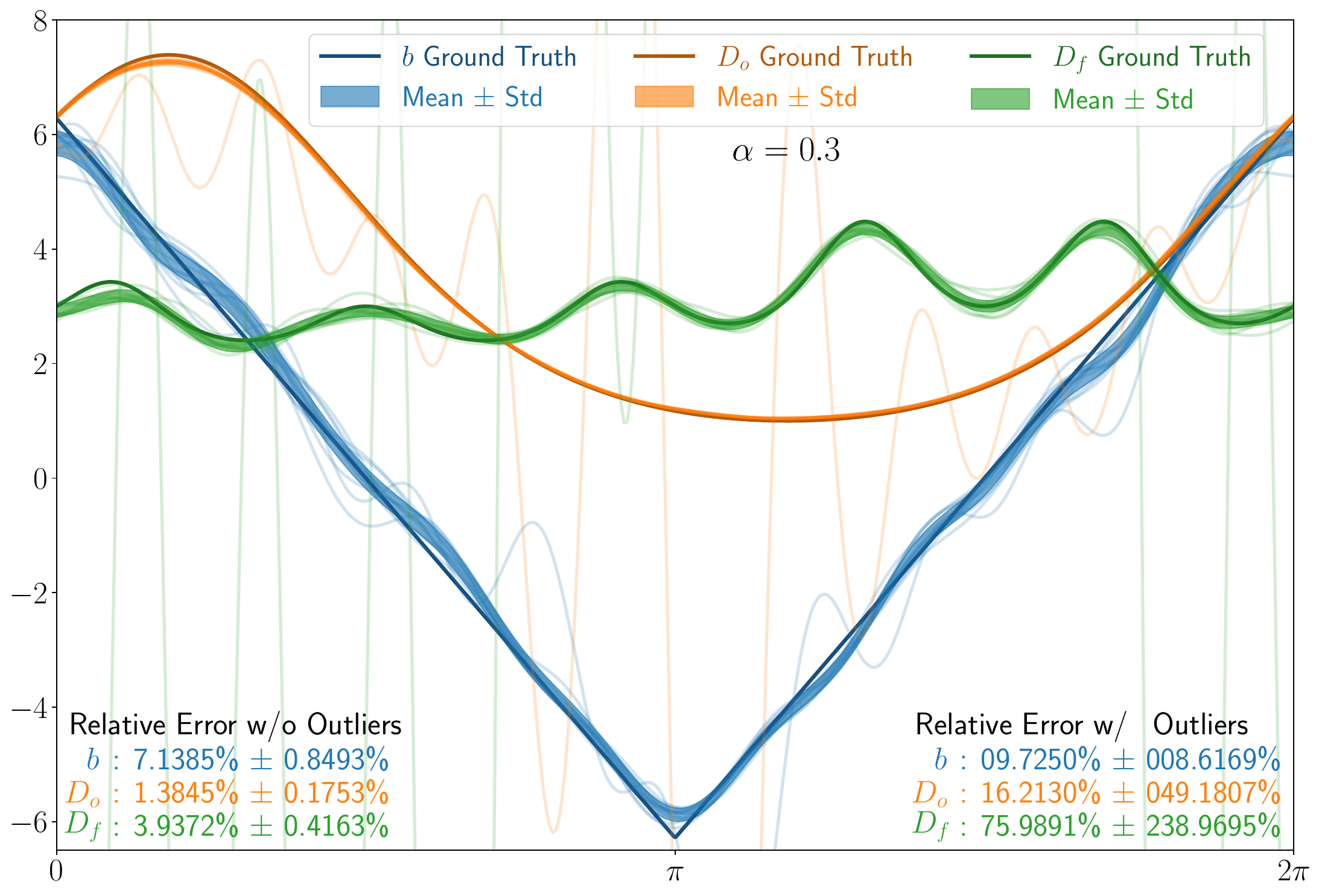}
    \caption
    {
        Left panel: Relative error (after removing the outliers) of $\drift, \diffusionOrdinary, \diffusionFractional$ versus the number of trajectories for $\alpha = 0.3$.
        Right panel: Recovery results for 100,000 trajectories for $\alpha = 0.3$, comparing predicted variables with ground truth.
        Related to \Cref{eg:unbiased_variable}.
    }
    \label{fig:unbiased_variable}
\end{figure}

To address the instability in recovery for heavy-tailed data, we filter out tail data from unbiased trajectory data, making it a type of censored trajectory data.
We introduce a new parameter, referred to as the tail removal threshold (\texttt{TRT}), which is used to filter out all pairs $(x_{\textnormal{current}}, x_{\textnormal{next}})$ that satisfy the condition $|x_{\textnormal{next}} - x_{\textnormal{current}} - \mu| \geq \texttt{TRT}$ from both sample pools, $P_{\textnormal{main}}$ and $P_{\textnormal{tail}}$.
A detailed description of the algorithm, including the implementation of \texttt{TRT}, is presented in \Cref{alg:Adam_with_tail_correction}.
In all experiments, we retain only tails within $\texttt{TRT} = 20$ and consider the remaining tails that exceed $\texttt{CT} = 8$ in the tail sample pool.
Next, we present numerical results on censored trajectory data and demonstrate the accuracy and robustness of our approach.

\subsubsection{Censored Trajectory Data}\label{sec:censored_data}
Two types of censored trajectory data are discussed in \Cref{sec:framework}.
More specifically, the filtering-based censored trajectory data is generated by first removing all jumps greater than or equal to $\texttt{TRT} = 20$ from the unbiased trajectory data.
Subsequently, a fixed seed, defined by the number of trajectories, is used to randomly discard half of the samples from the tail sample pool.
The corresponding examples are provided in \Cref{eg:censored_half_tail_constant,eg:censored_half_tail_variable}.
Another type of data is generated using random walk Metropolis-Hastings sampling with a burn-in number of 5,000, such that the exploration region of the data remains constrained within a limited range.
The corresponding example is provided in \Cref{eg:censored_MCMC_constant}.

\begin{example}\label{eg:censored_half_tail_variable}
    The ground truth coefficients are
    $\drift(x) = 4 |(x \text{ mod } 2 \pi) - \pi| - 2 \pi$,
    $\diffusionOrdinary(x) = \exp(\sin(x + 1) + 1)$,
    $\diffusionFractional(x) = 2 + \exp(\sin(2 x) \cos(3 x))$.
    We set $K = 21$.
    The time difference between consecutive points is $\Delta t = 1 / 40$.
    The censored trajectory data obtained through filtering are used in this case.
    A visualization of the results for $\alpha = 0.3$ is presented in \Cref{fig:censored_half_tail_variable}.
\end{example}

In \Cref{fig:censored_half_tail_variable}, we present results only for $\diffusionFractional$, as $\drift$ and $\diffusionOrdinary$ are largely unaffected by the tail correction technique. The results demonstrate the effectiveness of tail correction on censored trajectory data.
Separate graphs corresponding to the right panel in \Cref{fig:censored_half_tail_variable} are provided in \Cref{fig:comparison}.

\begin{figure}[htp]
    \centering
    \includegraphics[width=.48\linewidth]{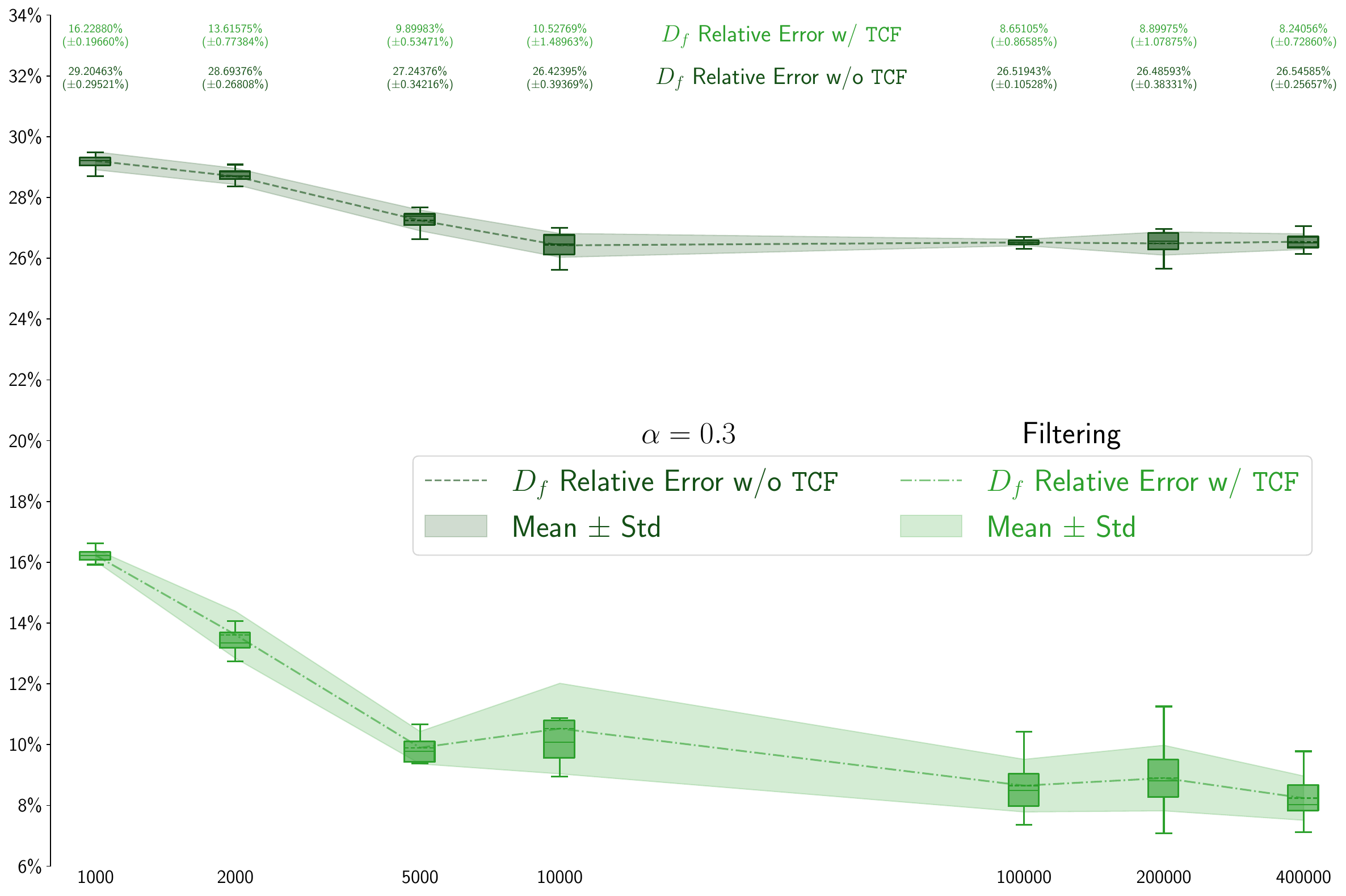}
    \hfill
    \includegraphics[width=.48\linewidth]{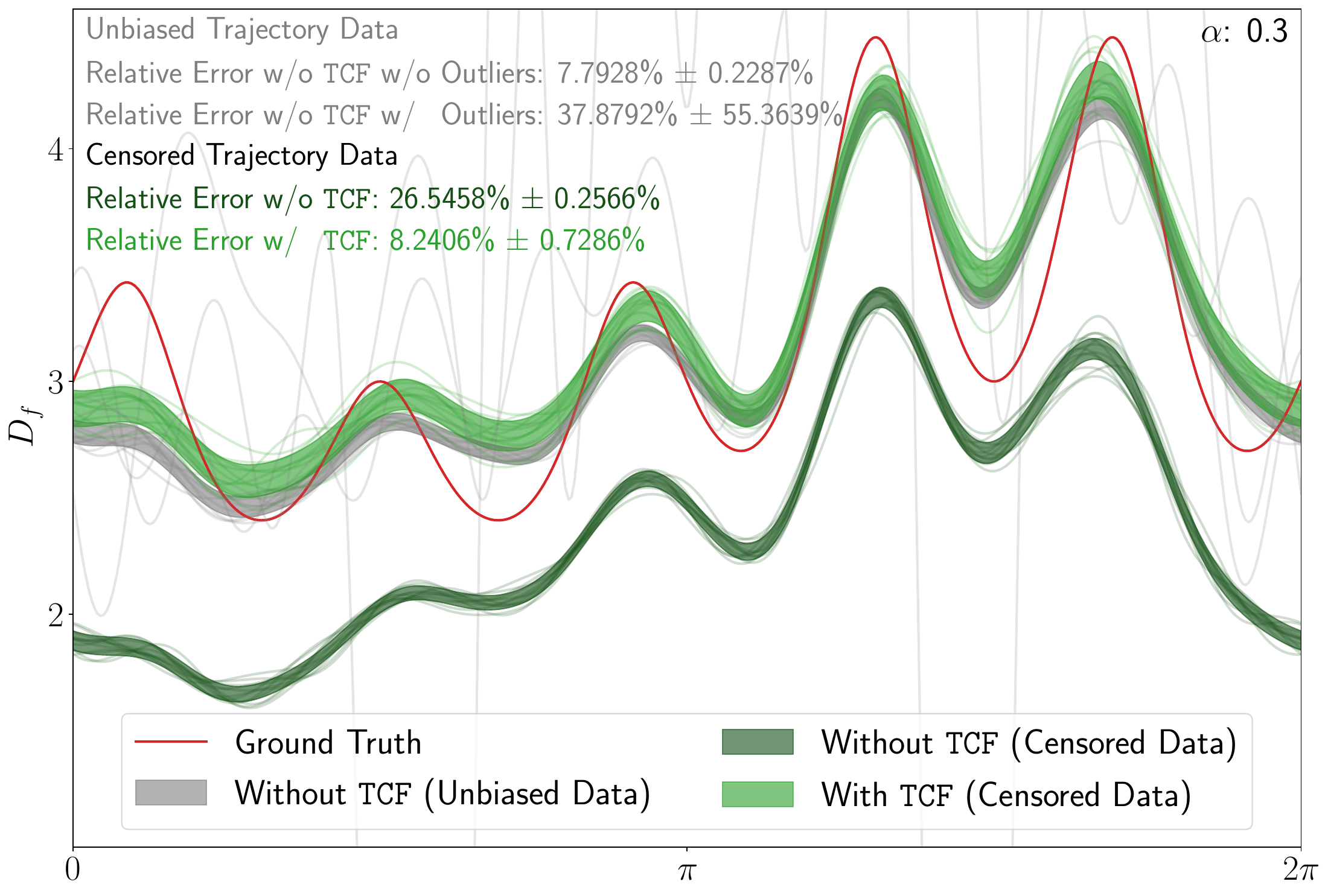}
    \caption
    {
        Left panel: Relative error of $\diffusionFractional$ versus the number of trajectories for $\alpha = 0.3$, comparing results with and without \texttt{TCF}.
        Right panel: Recovery results for 400,000 trajectories for $\alpha = 0.3$, showing the effect of \texttt{TCF}.
        Related to \Cref{eg:censored_half_tail_variable}.
    }
    \label{fig:censored_half_tail_variable}
\end{figure}

\begin{example}\label{eg:censored_half_tail_constant}
    The ground truth coefficients are $(\drift, \diffusionOrdinary, \diffusionFractional) = (5, 4, 3)$.
    We assume that we know the target coefficients are constant and set $K = 1$.
    The time difference between consecutive points is $\Delta t = 1 / 40$.
    The censored trajectory data obtained through filtering are used in this case.
    A visualization of the results for $\alpha = 0.3$ is provided in the left panel of \Cref{fig:censored_constant}.
\end{example}

\begin{example}\label{eg:censored_MCMC_constant}
    This example is similar to \Cref{eg:censored_half_tail_constant}, except that the censored trajectory data are generated using the MCMC sampler.
    A visualization of the results for $\alpha = 0.3$ is provided in the right panel of \Cref{fig:censored_constant}.
\end{example}

From \Cref{eg:censored_half_tail_constant,eg:censored_MCMC_constant}, it can be observed that the application of \texttt{TCF} significantly enhances the estimation accuracy of $\diffusionFractional$, although accompanied by increased variance.
The primary reason for this increased variance is that, at each step of gradient descent, sampling from the tail data pool occurs with probability \texttt{TCF}, generally a small value, introducing additional stochasticity into the gradient descent process.
With a sufficiently large dataset, the incorporation of \texttt{TCF} effectively mitigates methodological bias.
Additionally, based on \Cref{fig:censored_constant}, the threshold number of trajectories required to achieve accurate estimates for $\alpha = 0.3$ varies between the cases \Cref{eg:censored_half_tail_constant,eg:censored_MCMC_constant}.

\begin{figure}[htp]
    \centering
    \includegraphics[width=.48\linewidth]{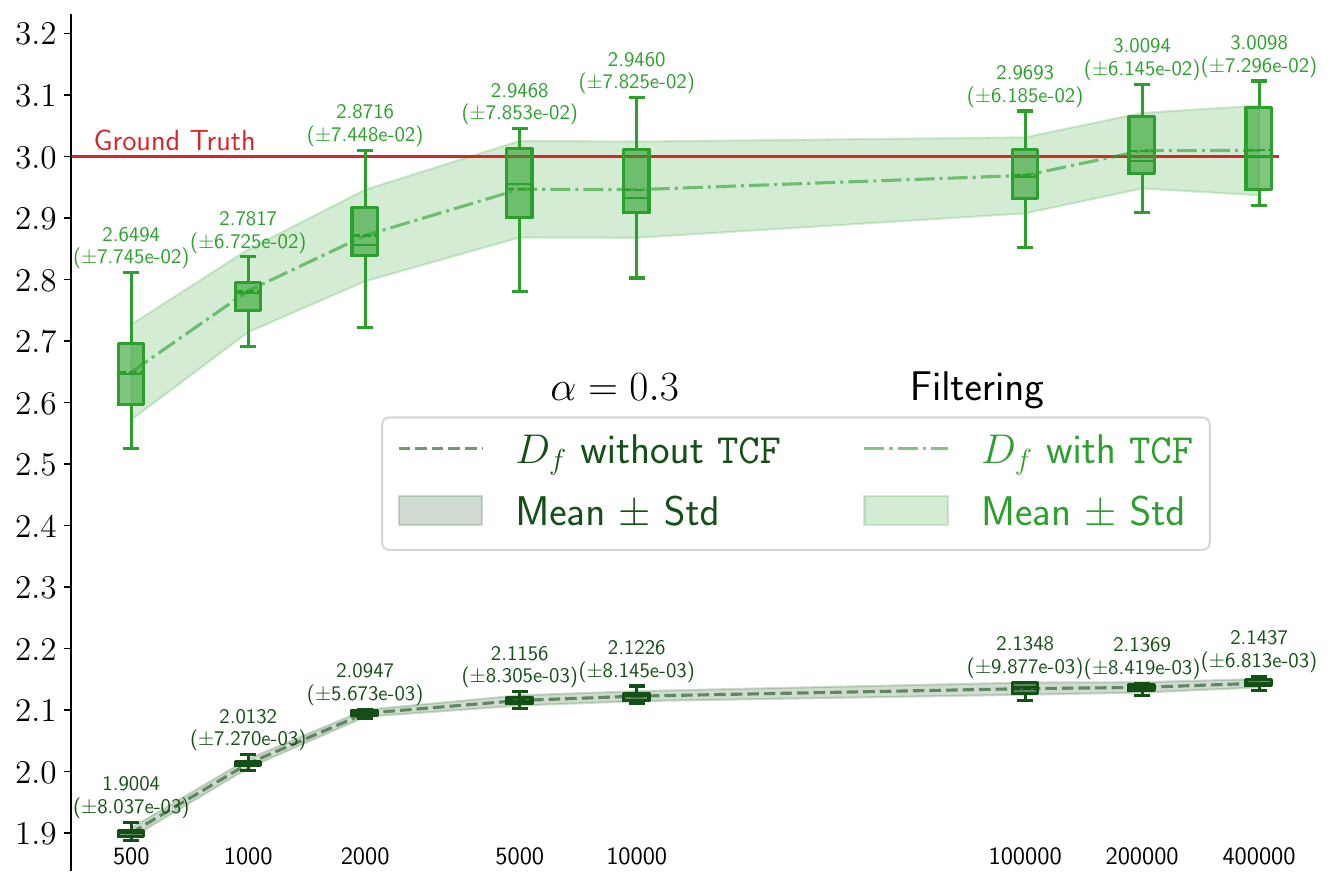}
    \hfill
    \includegraphics[width=.48\linewidth]{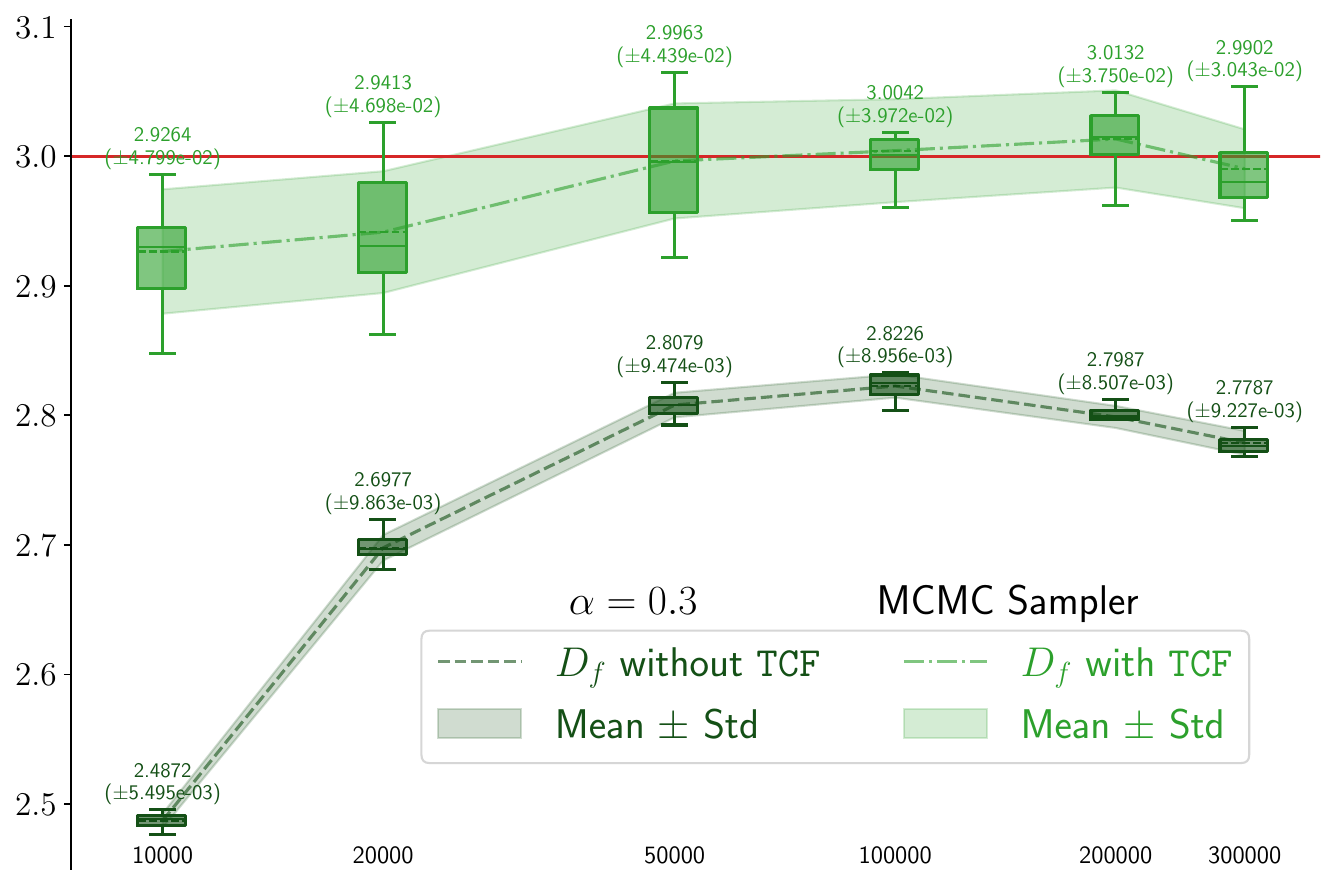}
    \caption
    {
        Left panel: Comparison of cases with and without \texttt{TCF} for $\alpha = 0.3$, using censored trajectory data obtained through filtering.
        Related to \Cref{eg:censored_half_tail_constant}.
        Right panel: Comparison of cases with and without \texttt{TCF} for $\alpha = 0.3$, using censored trajectory data generated by the MCMC sampler.
        Related to \Cref{eg:censored_MCMC_constant}.
    }
    \label{fig:censored_constant}
\end{figure}

\subsection{Policy Evaluation}\label{sec:policy_evaluation_experiment}
In this subsection, numerical experiments are conducted to illustrate the policy-evaluation error behavior described by \Cref{thm:error_bound}.
This illustration uses numerical examples from \Cref{eg:censored_half_tail_variable,eg:censored_half_tail_constant} in \Cref{sec:policy_evaluation_error}, as well as asymptotic results from a manufactured problem in \Cref{sec:asymptotic_linear_dependence}.
All experiments in this subsection aim to solve \Cref{eq:PDE_value_function_approx} with the parameter $\beta = 0.1$.
The manufactured solution is defined as $V(x) = \cos^{3}(2 x)$, from which the corresponding reward function $r(x)$ is derived using the ground truth coefficients $\drift$, $\diffusionOrdinary$, and $\diffusionFractional$.

\subsubsection{Policy Evaluation Error}\label{sec:policy_evaluation_error}
In this part, numerical experiments are conducted to illustrate the improvement in policy evaluation accuracy achieved through the tail correction technique.
The left panel of \Cref{fig:policy_evaluation_error} shows policy evaluation results obtained using the learned coefficients, alongside those computed with the ground truth coefficients, based on the setup described in \Cref{eg:censored_half_tail_variable}.
All 12 independent coefficient recovery results corresponding to the scenario with 400,000 trajectories are utilized in the figure.
Note that for unbiased data, three outliers were omitted from the graph on the left panel because they fall far outside the plotted domain.

\begin{figure}[htp]
    \centering
    \includegraphics[width=.48\linewidth]{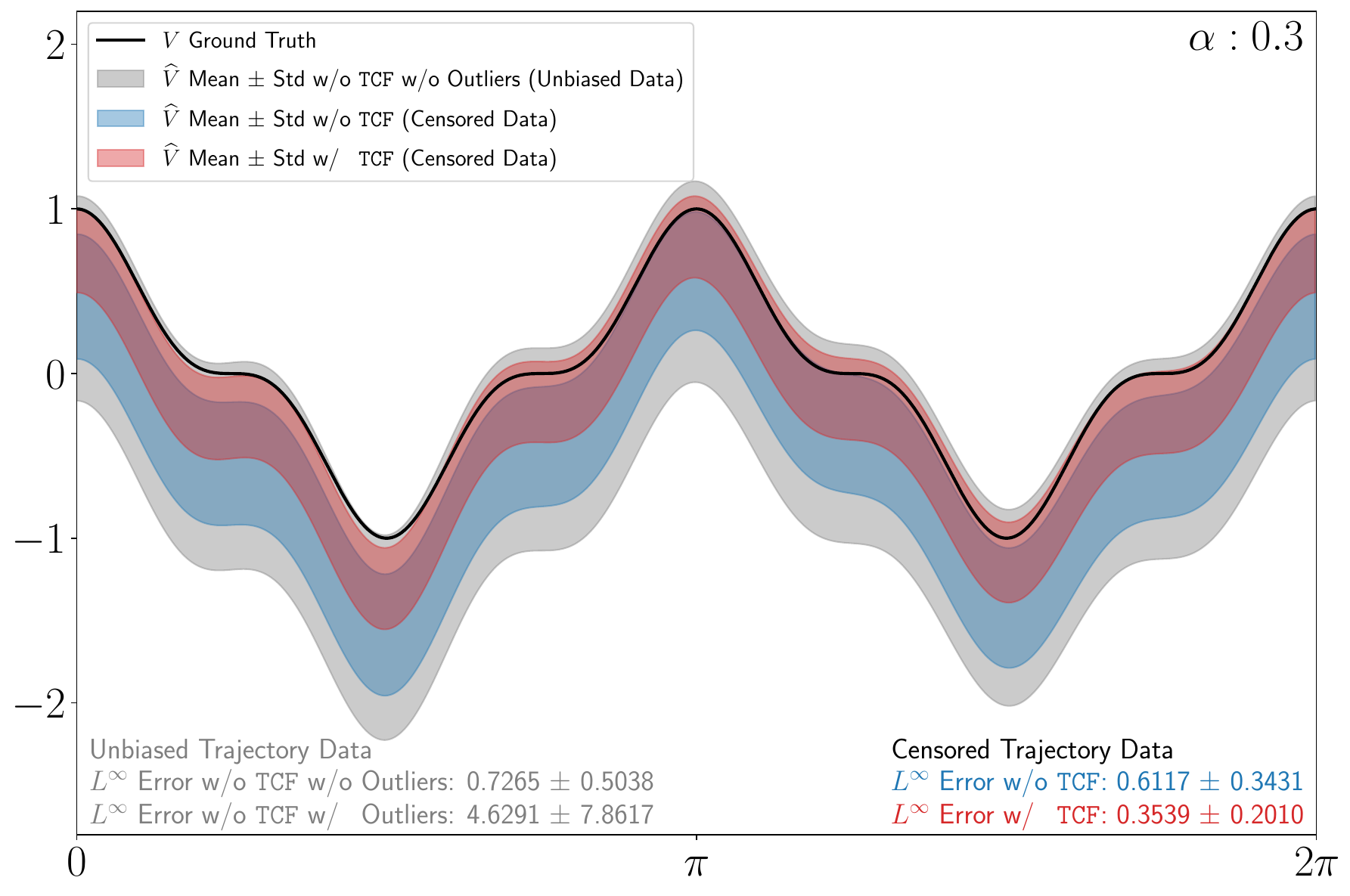}
    \hfill
    \includegraphics[width=.48\linewidth]{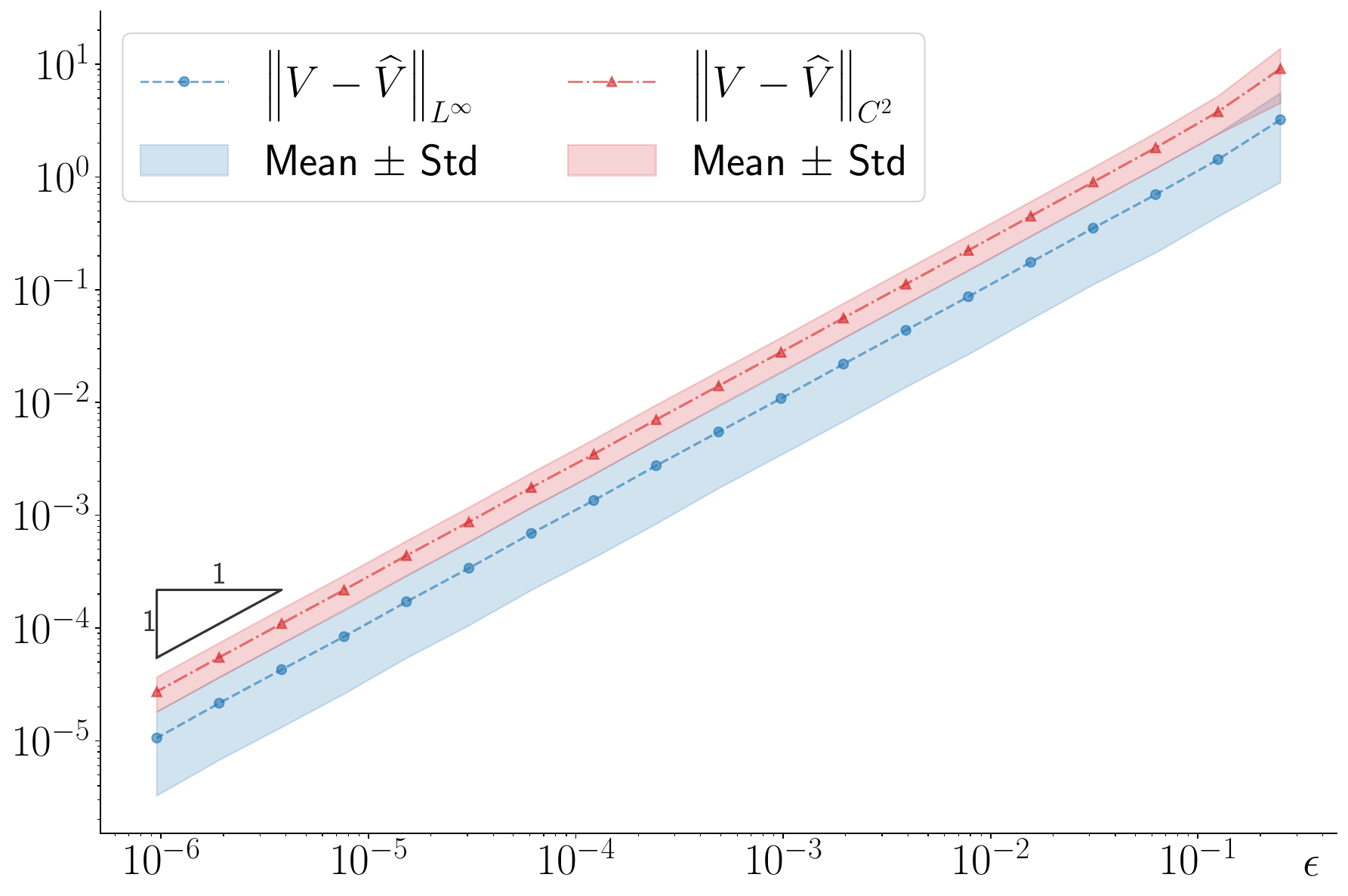}
    \caption
    {
        Left panel: Comparison of the computed value functions using results obtained from \Cref{eg:censored_half_tail_variable}, where the ground truth involves variable coefficients.
        For the unbiased-data curve, three outlying learned-coefficient runs are omitted from the plotted range and are discussed in the text.
        Right panel: Asymptotic linear dependence of the policy evaluation error on the estimation error in a manufactured problem with different perturbation magnitudes.
    }
    \label{fig:policy_evaluation_error}
\end{figure}

As demonstrated in \Cref{fig:policy_evaluation_error}, the policy evaluation using coefficients recovered with the tail correction technique yields improved performance. 

\subsubsection{Asymptotic Rate Study}\label{sec:asymptotic_linear_dependence}
In this part, numerical experiments are conducted to investigate the asymptotic linear dependence of the policy evaluation error on the estimation error as established in \Cref{thm:error_bound}.
Specifically, we consider a manufactured problem with coefficients chosen as follows: $\drift(x) = \sin^{4}(x)$, $\diffusionOrdinary(x) = \cos^{2}(x) + |\sin(x)|$ and $\diffusionFractional(x) = \sin(4 x) + 2$.

To numerically verify the linear dependence on the estimation error parameter $\epsilon$, we perturb each coefficient independently by adding Gaussian noise with distribution $\mathcal{N}(0, \epsilon)$, repeating the procedure for 10,000 trials at each chosen value of $\epsilon$. 
The numerical results are illustrated in the right panel of \Cref{fig:policy_evaluation_error}.

Alternative perturbation strategies, such as adding or subtracting Gaussian- or wedge-shaped functions with amplitude proportional to $\epsilon$, produce qualitatively similar outcomes.
Hence, we present only one representative scenario in this study.

\subsection{Real-Data Experiment}\label{sec:real_data_experiment}
In addition to the synthetic experiments, we conduct a real-data experiment using the Bitcoin (BTC) price as a representative continuously traded asset.
This choice avoids artificial day-gap effects caused by overnight market closures.
Data are obtained through IBKR API at 3-minute resolution, covering August 15, 2021 to March 11, 2026, with 634,551 observations.
The Bitcoin open price ranges from approximately 15,561 USD to 126,078 USD over this period.
Using this empirical trajectory, we construct discounted empirical values for the reward function $r(x) = \cos^{ 3 } \left ( 2 \pi ( x + 128000 ) / 256000 \right )$ with discount parameter $\beta = 0.3$.
For each starting index $i$, the empirical value is
\begin{equation*}
    \widetilde{V}( x_{ i } )
    =
    \Delta t
    \sum_{ j = 0 }^{ 3000 }
    \exp( - \beta j \Delta t )
    r( x_{ i + j } ),
\end{equation*}
where $\Delta t = 1 / 20$ is the 3-minute sampling interval measured in hours.
As a reference curve, $\widehat{ \widetilde{V} }$ is obtained by a least-squares Fourier approximation to these empirical values.
For each estimated generator
\begin{equation*}
    \widehat{\mathcal{ L }}
    :=
    \widehat{\drift}(x) \cdot \nabla
    +
    \widehat{\diffusionOrdinary}(x) : \nabla^{2}
    -
    \widehat{\diffusionFractional}(x) (- \Delta)^{\widehat{\alpha}},
\end{equation*}
the ODE-based value function is obtained from the resolvent equation
\begin{equation*}
    ( \beta I - \widehat{ \mathcal{ L } } ) [ \widehat{ V } ]( x )
    =
    r( x ).
\end{equation*}
We approximate $\widehat{ V }( x )$ using Fourier basis as
\begin{equation*}
    \widehat{ V }^{ K }( x )
    =
    \sum_{ i = 1 }^{ K } c_{ i } \phi_{ i }( x ).
\end{equation*}
The coefficients $\{ c_{ i } \}_{ i = 1 }^{ K }$ are not fitted directly to empirical policy-evaluation values.
Instead, they are chosen so that $\widehat{ V }^{ K }$ satisfies the ODE as accurately as possible on the data-supported window $\mathcal{ D }$.
This gives the least-squares problem
\begin{equation*}
    \min_{ \{ c_{ i } \}_{ i = 1 }^{ K } }
    \int_{ \mathcal{ D } }
    \left (
        \sum_{ i = 1 }^{ K } c_{ i }
        ( \beta I - \widehat{ \mathcal{ L } } )[ \phi_{ i } ]( x )
        -
        r( x )
    \right )^{ 2 }
    \, d x.
\end{equation*}
Thus the policy-evaluation curve fits discounted values computed from the observed trajectory, while the ODE curve fits the residual of the estimated generator equation.
Unlike the preceding synthetic experiments, where $\alpha$ is treated as known and fixed, this real-data experiment also learns $\alpha$ from the BTC trajectory.
In this subsection, we use $11$ Fourier basis functions for each variable coefficient and $\alpha$ is learned as a constant.
The corresponding gradient formula follows from the expansion of $\alpha(x)$ described in the remark in \Cref{appendix:numerical_method}.
\Cref{fig:real_data_advantages} compares the ODE-based value functions obtained from the learned generators with the empirical reference $\widehat{ \widetilde{V} }$.
In the left panel, the comparison uses the first $6000$ transitions and includes classical reinforcement learning (CRL).
Here, CRL means a direct model-free baseline that fits the discounted empirical values from the observed transitions without estimating the L\'evy generator and without applying the \texttt{TCF} correction.
In the right panel, the comparison uses fixed random $60000$ transitions.
The orange points are $5000$ samples from all empirical values $\widetilde{V}( x_{ i } )$, and the black curve is the reference $\widehat{ \widetilde{V} }$.
The red and blue bands show the mean $\pm$ one standard deviation, based on $8$ independent tests, of the selected ODE-based approximations $\widehat{V}$ with and without \texttt{TCF}, respectively.
The lower panel plots the corresponding differences from $\widehat{ \widetilde{V} }$, and the marked percentage reports the fraction of empirical values covered by the mean $\pm$ standard-deviation band.
The numerical outputs in \Cref{fig:real_data_advantages} are reported as relative $L^{2}$ errors and empirical coverage percentages.
In the first-$6000$-transition experiment, adding \texttt{TCF} improves the relative $L^{2}$ error from $12.04 \pm 11.83\%$ to $11.86 \pm 8.46\%$ and increases the coverage from $65.91\%$ to $99.68\%$.
The mean error improvement in this small-data test is modest, but the coverage increases even as the standard deviation of the error decreases.
In contrast, the same-data CRL baseline develops large oscillatory extrapolation errors outside the observed price range.
This indicates an advantage of using the learned stochastic generator: the method is not restricted to pointwise fitting of observed values, but can propagate information through the estimated differential operator.
In the random-$60000$-transition experiment, \texttt{TCF} reduces the relative $L^{2}$ error from $4.48 \pm 1.83\%$ to $3.76 \pm 0.92\%$ and increases the coverage from $21.80\%$ to $61.42\%$.
Thus, the coverage gain is not obtained by merely widening the uncertainty band; in both panels the coverage rises while the reported error standard deviation drops.
Overall, the main improvement is the comparison between the \texttt{TCF}-corrected generator and the same generator trained without \texttt{TCF}.
The smaller error in the random-$60000$ panel also suggests that using more transitions sampled across the price range improves the learned generator, although this comparison changes the data-selection protocol as well as the sample size.
The left lower panel marks in green the price range covered by the first $6000$ transitions.

\begin{figure}
    \centering
    \includegraphics[width=0.48\linewidth]{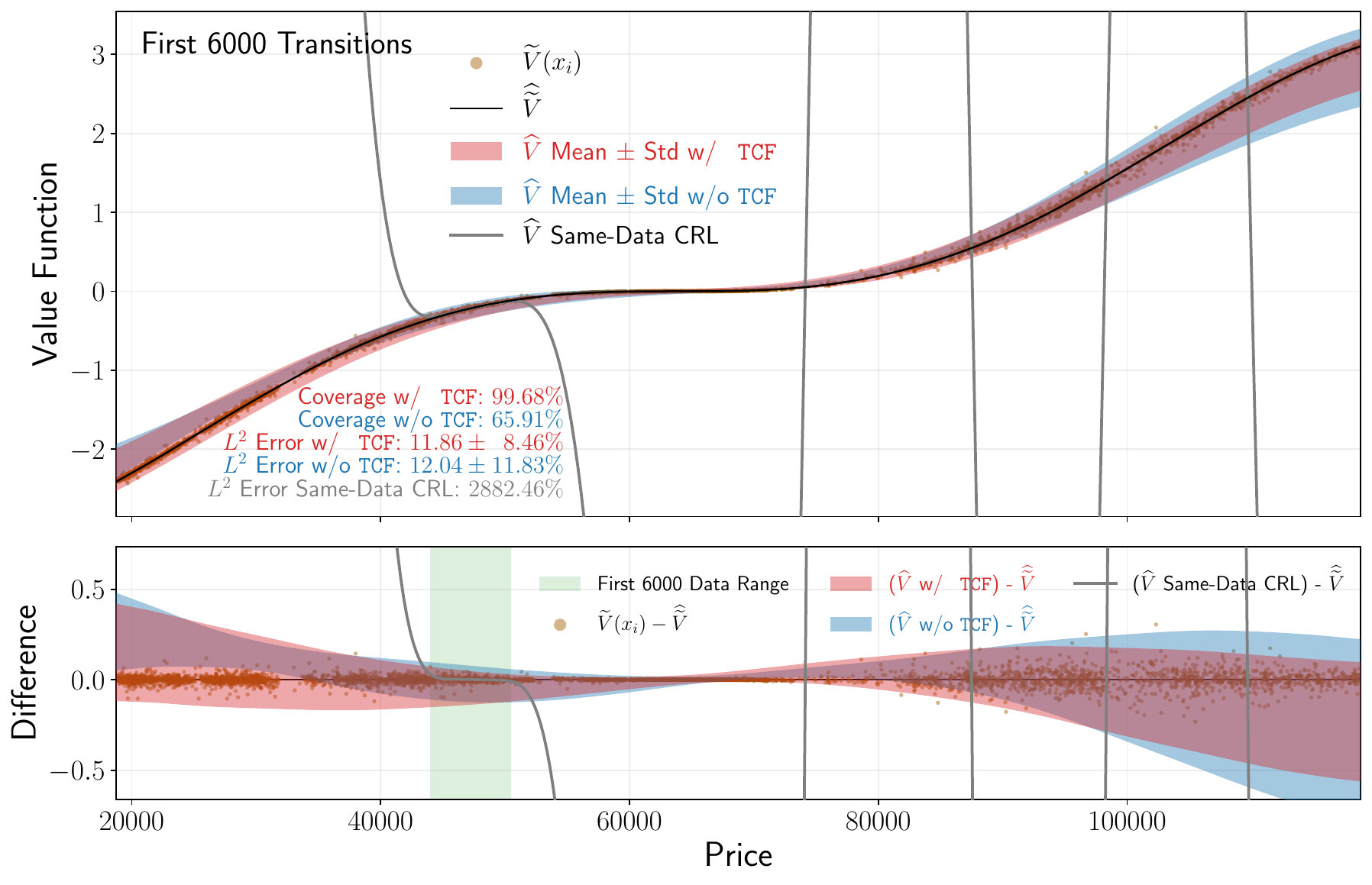}
    \hfill
    \includegraphics[width=0.48\linewidth]{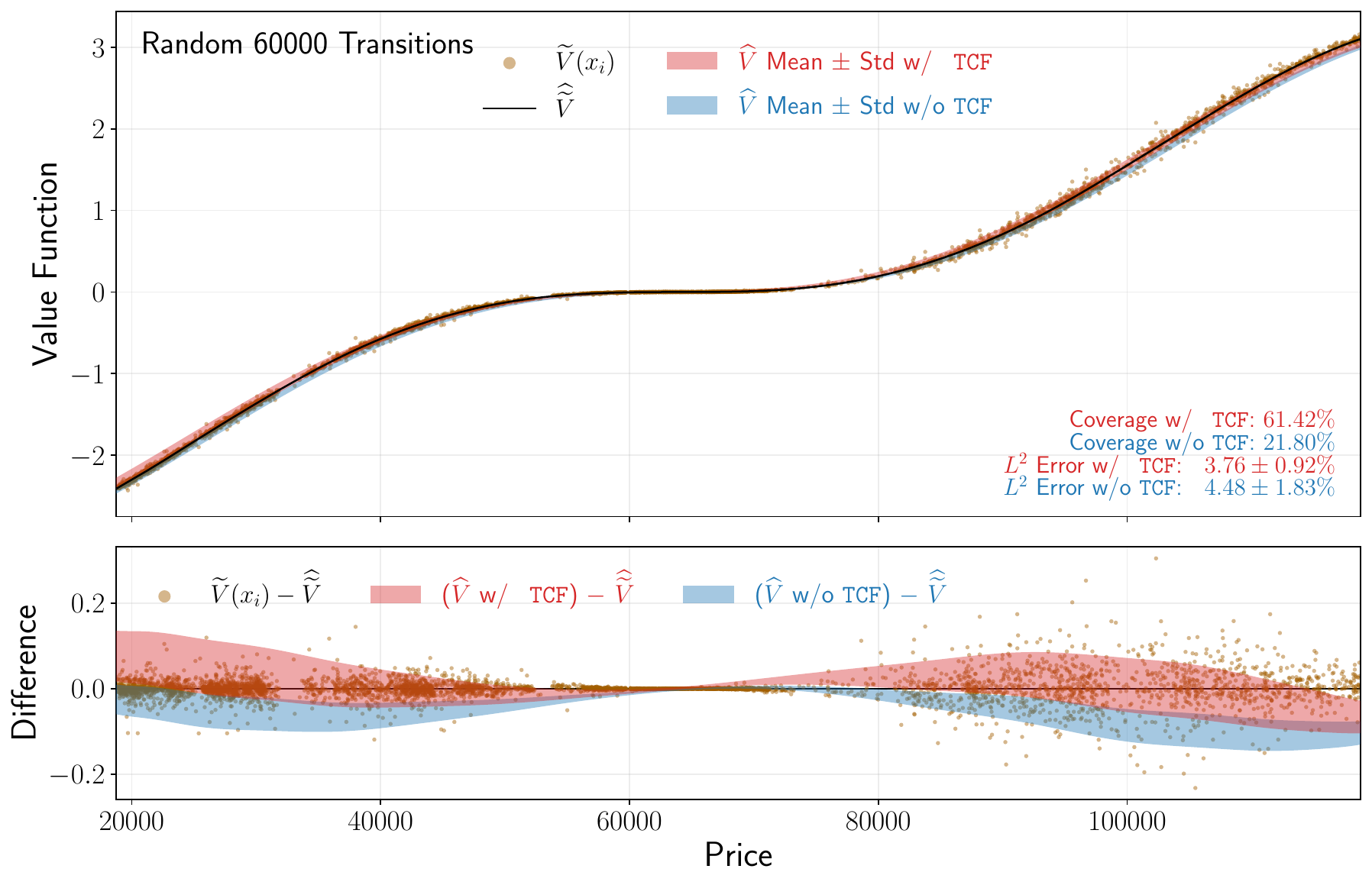}
    \caption{
        Real-data BTC policy-evaluation comparison.
        The orange dots show $5000$ samples from all empirical policy-evaluation values.
        The \texttt{TCF}-corrected learned generator tracks the empirical reference more closely than both the learned generator without \texttt{TCF} and classical reinforcement learning, i.e., directly fitting discounted empirical values without learning the L\'evy generator.
        The left panel shows that the approach remains usable with only the first $6000$ transitions, while CRL has large extrapolation errors outside the observed range.
        The right panel shows that using random $60000$ transitions improves accuracy and stability.
        In both panels, \texttt{TCF} improves both coverage and $L^{2}$ error relative to the policy evaluation without \texttt{TCF}.
    }
    \label{fig:real_data_advantages}
\end{figure}


\section{Conclusion}\label{sec:conclusion}
This work introduces a model-based continuous-time policy evaluation framework where the underlying stochastic dynamics incorporates both Brownian and L\'evy noise.
This extends traditional models that rely solely on Brownian dynamics, providing a more realistic representation of stochastic environments encountered in various real-world applications.
A key contribution is the development of an accurate and robust numerical method for recovering L\'evy dynamics, particularly in cases where heavy-tailed behavior is pronounced, such as when the fractional exponent $\alpha$ is small.
Additionally, this paper establishes a theoretical bound for policy evaluation errors based on the recovery error of coefficients in stochastic dynamics.
This indicates that the accuracy of policy evaluation depends on both the recovery error of the stochastic dynamics and the numerical error in solving the associated partial integro-differential equation (PIDE).
These findings contribute to a more reliable and mathematically rigorous foundation for reinforcement learning in complex stochastic systems.

Future research can build on this work in several directions.
A deeper analysis of coefficient recovery error from discrete-time trajectory data would provide further insights into the accuracy of the proposed approach.
Additionally, understanding the convergence of the iterative tail correction method remains an important theoretical challenge.
Developing an efficient numerical solver for the fractional Fokker--Planck equation with variable coefficients is another key objective, as it would improve the accuracy of stochastic dynamics recovery.
Moreover, systematically inferring the fractional exponent $\alpha$, especially when it is state-dependent, would make the approach more adaptable to real-world applications.
Finally, extending these techniques to higher-dimensional settings and broader real-world datasets will be crucial for validating their effectiveness in practical applications.


\section*{Acknowledgement}
Q.~Ye and X.~Tian were supported in part by NSF DMS-2240180 and the Alfred P. Sloan Fellowship.
Y.~Zhu was supported in part by NSF grant DMS-2529107.

The authors would like to express sincere gratitude to Boris Baeumer, Yuan Chen, Bin Dong, Qiang Du, Bo Li, Pearson Miller, Mingtao Xia, Yuming Paul Zhang, Yi Zhu for their insightful discussions.
The authors used GPT-5.5 only to check grammar and typos.
The tool did not change the meaning or content of the manuscript.
The authors assume responsibility for all content.


\bibliographystyle{abbrv}
\bibliography{bib/ref}


\appendix


\section{Selecting the Cutting Threshold}\label{appendix:CT_selection}

This appendix provides a practical rule for estimating the size of the cutting threshold \texttt{CT} used in \Cref{eq:tail_pool}.
The goal is not to determine an exact threshold, but to choose a scale that separates the central part of the one-step increment distribution from the tail part.

We first estimate an average value of the fractional diffusion coefficient.
For example, after a preliminary coefficient recovery step, one can take
\begin{equation*}
    \overline{\diffusionFractional}
    =
    \frac{1}{|P_{\textnormal{main}}|}
    \sum_{\left ( x_{\textnormal{current}}, x_{\textnormal{next}} \right ) \in P_{\textnormal{main}}}
    \widehat{\diffusionFractional}(x_{\textnormal{current}}).
\end{equation*}
For a Fourier basis, this average is close to the constant Fourier mode when the observed states cover the domain sufficiently well.

To estimate the threshold scale, we use the one-dimensional constant-coefficient density in \Cref{eq:FFPE_constant} and simplify it by setting
\begin{equation*}
    x_{0} = 0, \quad
    \drift = 0, \quad
    \diffusionOrdinary = 0, \quad
    \diffusionFractional = \overline{\diffusionFractional}.
\end{equation*}
For one time step, let
\begin{equation*}
    \tau = \overline{\diffusionFractional} \Delta t.
\end{equation*}
Then the reduced density is
\begin{equation}\label{eq:CT_reduced_density}
    p(x; \alpha, \tau)
    =
    \frac{1}{\pi}
    \int_{0}^{\infty}
    \cos(x r) \exp \left ( - \tau r^{2 \alpha} \right ) \, d r.
\end{equation}
For a prescribed central mass $R \in (0, 1)$, define $C(R; \alpha, \tau)$ by
\begin{equation}\label{eq:CT_central_mass}
    \int_{- C(R; \alpha, \tau)}^{C(R; \alpha, \tau)}
    p(x; \alpha, \tau) \, d x
    =
    R.
\end{equation}
The density in \Cref{eq:CT_reduced_density} is a symmetric $2 \alpha$-stable density with scale $\tau^{1 / (2 \alpha)}$.
If $F_{2 \alpha}$ denotes the cumulative distribution function of the unit-scale symmetric $2 \alpha$-stable law with characteristic function $\exp \left ( - |k|^{2 \alpha} \right )$, then
\begin{equation}\label{eq:CT_formula}
    C(R; \alpha, \overline{\diffusionFractional} \Delta t)
    =
    \left ( \overline{\diffusionFractional} \Delta t \right )^{1 / (2 \alpha)}
    F_{2 \alpha}^{-1}
    \left ( \frac{1 + R}{2} \right ).
\end{equation}
Therefore, a practical choice is
\begin{equation*}
    \texttt{CT}
    \approx
    \left ( \overline{\diffusionFractional} \Delta t \right )^{1 / (2 \alpha)}
    F_{2 \alpha}^{-1}
    \left ( \frac{1 + R}{2} \right ).
\end{equation*}
With this choice, the interval $[- \texttt{CT}, \texttt{CT}]$ contains approximately the central $R$ fraction of the one-step increments under the reduced stable model.
Equivalently, about $1 - R$ of the reduced model lies in the tail region.

We recommend taking $R \geq 0.9$ so that only relatively large jumps are included in $P_{\textnormal{tail}}$.
In most of our practices, we take the more conservative value $R = 0.98$ when estimating the scale of \texttt{CT}.

\section{Full Version of \texorpdfstring{\Cref{alg:Adam_with_tail_correction_concise}}{Algorithm 3.1}}\label{appendix:algorithm}
\begin{algorithm}[htbp]
    \caption{Robust Maximum Likelihood Estimation}\label{alg:Adam_with_tail_correction}
    \begin{algorithmic}[1]
        \Require All the trajectories $P = \left \{ x_{j \Delta t}^{(i)} \right \}_{j = 0, i = 1}^{j = J, i = I}$, the cutting threshold \texttt{CT} and the tail removal threshold \texttt{TRT}
        \Ensure $\theta$ representing the estimated coefficients

        \State $\eta, \beta_{1}, \beta_{2}, \epsilon \gets \textit{1e-2}, 0.9, 0.999, \textit{1e-8}$
            \Comment{Adjustable hyperparameters for Adam}
        \State $I, B \gets 40000, 100$
            \Comment{Adjustable step limit and batch size}
        \State $\boldsymbol{m}, \boldsymbol{v}, i, \texttt{TCF} \gets \boldsymbol{0}, \boldsymbol{0}, 0, 0$
            \Comment{Initialization}
        \State Obtain $\Delta t$ from data and randomly initialize $\theta$ to make sure $\diffusionOrdinary, \diffusionFractional \geq 0$
        \State $P_{\textnormal{main}} \gets \left \{ \left ( x_{j \Delta t}^{(i)}, x_{(j + 1) \Delta t}^{(i)} \right ) \right \}_{j = 0, i = 1}^{j = J - 1, i = I}$
            \Comment{The main sample pool}
        \State $\mu \gets \text{median} \left ( \Delta x \Big | \left ( x_{\textnormal{current}}, x_{\textnormal{next}} \right ) \in P_{\textnormal{main}} \right )$
            \Comment{$\Delta x := x_{\textnormal{next}} - x_{\textnormal{current}}$}
        \State $P_{\textnormal{main}} \gets \left \{ \left ( x_{\textnormal{current}}, x_{\textnormal{next}} \right ) \in P_{\textnormal{main}} \bigg | \left | \Delta x - \mu \right | < \texttt{TRT} \right \}$
        \State $P_{\textnormal{tail}} \gets \left \{ \left ( x_{\textnormal{current}}, x_{\textnormal{next}} \right ) \in P_{\textnormal{main}} \bigg | \left | \Delta x - \mu \right | > \texttt{CT} \right \}$
            \Comment{The tail part sample pool, if its size is too small, reduce \texttt{CT} by half until its size is large enough (at least $\geq B$)}
        \While{$i < I$ \textbf{and} the moving average of $\theta$ is not converged}
            \State $i \gets i + 1$
            \If{a random number $\sim \text{Uniform}(0, 1)$ exceeds \texttt{TCF}}
                \State $\ell(\theta)$ will be computed by $B$ samples from $P_{\textnormal{main}}$
                \Comment{Normal sampling}
            \Else
                \State $\ell(\theta)$ will be computed by $B$ samples from $P_{\textnormal{tail}}$
                \Comment{Tail sampling}
            \EndIf
            \State $\boldsymbol{g} \gets$ the gradient of $- \ell(\theta)$
                \Comment{\Cref{alg:FFPE_direct_gradient} or \Cref{alg:FFPE_finite_difference_gradient} or their mixture}
            \State $\boldsymbol{m}, \boldsymbol{v} \gets \beta_{1} \boldsymbol{m} + (1 - \beta_{1}) \boldsymbol{g}, \beta_{2} \boldsymbol{v} + (1 - \beta_{2}) \boldsymbol{g} \odot \boldsymbol{g}$
                \Comment{Update moments}
            \State $\widehat{\boldsymbol{m}}, \widehat{\boldsymbol{v}} \gets \boldsymbol{m} / (1 - \beta_{1}^{i}), \boldsymbol{v} / (1 - \beta_{2}^{i})$
                \Comment{Correct the bias}
            \State $\theta \gets \theta - \eta \widehat{\boldsymbol{m}} \odot (\sqrt{\widehat{\boldsymbol{v}}} + \epsilon)^{- 1}$
                \Comment{Update $\theta$}
            \State $\texttt{TCF} \gets $ \Cref{alg:TCF_computation}($\theta$, $\Delta t$, \texttt{CT}, $\mu$, $|P_{\textnormal{tail}}| / |P_{\textnormal{main}}|$, $\alpha$) \textbf{if} $i \geq 4000$
                \Comment{Update \texttt{TCF}}
        \EndWhile
            \Comment{Concise version can be found at \Cref{alg:Adam_with_tail_correction_concise}}
        \State \Return the moving average of $\theta$
    \end{algorithmic}
\end{algorithm}

\section{Proof of \texorpdfstring{\Cref{thm:error_bound}}{Theorem 3.1}}\label{sec:proof_of_theorem}
In this section, we provide the proof of \Cref{thm:error_bound}. 
The central idea relies on the regularity results for viscosity solutions of elliptic integro-differential equations.
Before we prove the theorem, we first present several important results related to the well-posedness and regularity of elliptic integro-differential equations. See \cite{barles2008second,mou2021regularity} for more details.

Denote $F(x, u(x), \nabla u(x), \nabla^{2} u(x), - (- \Delta)^{\alpha} u(x)) = \beta u (x) - r(x) - \drift(x) \cdot \nabla u (x) - \diffusionOrdinary(x) : \nabla^{2} u (x) + \diffusionFractional(x) (- \Delta)^{\alpha} u(x)$. 
Then \Cref{eq:PDE_value_function} is equivalent to solving:
\begin{equation}
\label{eq:PDE_value_function_2}
    F(x, u(x), \nabla u(x), \nabla^{2} u(x), - (- \Delta)^{\alpha} u(x))
    = 0.
\end{equation}
For the remainder of this section, we assume that $r(x), b(x), \diffusionOrdinary(x), \diffusionFractional(x)$ are uniformly continuous and periodic functions defined on $\mathbb{R}^{d}$, with periodicity defined on the unit cell $Q = (0, 2 \pi]^{d}$.
We use the standard definition of viscosity solutions for elliptic equations, which is briefly quoted below.
\begin{definition}
\label{defn:viscositysolu}
    A bounded upper semicontinuous (resp. lower semicontinuous) function $u: \mathbb{R}^{d} \to \mathbb{R}$ is a viscosity subsolution (resp. supersolution) of \Cref{eq:PDE_value_function_2} if, for any test function
    $\phi \in C_{b}(\mathbb{R}^{d}) \cap C^{2}(\mathbb{R}^{d})$, if $x$ is a global maximum (resp. minimum) point of $u - \phi$, then
    \begin{equation*}
        F(x, u(x), \nabla u(x), \nabla^{2} u(x), - (- \Delta)^{\alpha} u(x))
        \leq 0
    \end{equation*}
    (resp. $F(x, u(x), \nabla u(x), \nabla^{2} u(x), -(- \Delta)^{\alpha} u(x)) \geq 0$).
    In addition, $u$ is called a viscosity solution if it is both a viscosity subsolution and supersolution.
\end{definition}

The following is a key comparison principle, as a result of \cite[Theorem 3]{barles2008second}.
\begin{lemma}
\label{lem:comparison_principle}
    If $u$ is a viscosity subsolution and $v$ is a viscosity supersolution, as defined in \Cref{defn:viscositysolu}, then $u \leq v$ in $\mathbb{R}^{d}$.
\end{lemma}

The following regularity result is a restatement of \cite[Theorems 4.1 and 5.1]{mou2021regularity}.
In the following theorem, $B_{1}$ and $B_{2}$ denote the open balls of radius $1$ and $2$, respectively.
However, we note that the theorem holds true, with constants adjusted, if $B_{1}$ and $B_{2}$ are any two open balls with $B_{2}$ containing $B_{1}$.
\begin{lemma}
\label{lem:Mou_Zhang}
    Let $u$ be a viscosity solution as defined in \Cref{defn:viscositysolu}.
    Then there exist constants $\gamma \in (0, 1)$ and $C_{1} > 0$ such that
    \begin{equation*}
        \| u \|_{C^{1, \gamma}(B_{1})}
        \leq C_{1} \left( \| u \|_{L^{\infty}(\mathbb{R}^{d})} + \| r \|_{L^{\infty}(B_{2})} \right).
    \end{equation*}
    In addition, assume that $r(x), b(x), \diffusionOrdinary(x), \diffusionFractional(x)$ are $\gamma$-H\"older where $\gamma \in (0, 1)$ is a sufficiently small universal constant. If $u \in C^{2}(B_{2}) \cap C^{0, \gamma}(\mathbb{R}^{d})$ is a classical solution in $B_2$, then there exists $C_{2} > 0$ such that 
    \begin{equation*}
        \| u \|_{C^{2, \gamma}(B_{1})}
        \leq C_{2} \left( \| u \|_{C^{0, \gamma}(\mathbb{R}^{d})} + \| r \|_{C^{0, \gamma}(B_{2})} \right).
    \end{equation*}
\end{lemma}
We note that \cite{mou2021regularity} made a slightly different regularity assumption on the nonlocal integral term. Following the proof of \cite[Theorem 5.1]{mou2021regularity}, we see that $\gamma$-H\"older continuity on $\diffusionFractional(x)$ is sufficient for the second statement in the above lemma to hold. 

We now establish two additional results that are essential for the proof of the theorem.
\begin{lemma}
\label{lem:regularity}
    Let $r(x), b(x), \diffusionOrdinary(x), \diffusionFractional(x)$ be uniformly continuous and periodic functions defined on $\mathbb{R}^{d}$, with periodicity defined on the unit cell $Q = (0, 2 \pi]^{d}$.
    Then \Cref{eq:PDE_value_function} has a unique viscosity solution $V$, where $V$ is a periodic function with the same periodicity $Q$ that satisfies
    \begin{equation}
    \label{eq:appendixB_regularity_1}
        \| V \|_{C^{1, \gamma}(\mathbb{R}^{d})}
        \leq C_{1} \| r \|_{L^{\infty}(\mathbb{R}^{d})},
    \end{equation}
    for some $\gamma\in (0, 1)$ and $C_{1} > 0$.
    In addition, suppose $r(x), b(x), \diffusionOrdinary(x), \diffusionFractional(x)$ are $\gamma$-H\"older continuous for a sufficiently small universal constant $\gamma$, then there exists $C_{2} > 0$ such that
    \begin{equation}
    \label{eq:appendixB_regularity_2}
        \| V \|_{C^{2, \gamma}(\mathbb{R}^{d})}
        \leq C_{2} \| r \|_{C^{0, \gamma}(\mathbb{R}^{d})}. 
    \end{equation}
\end{lemma}
\begin{proof}
    Existence comes from the classical Perron's method for viscosity solutions while uniqueness is as a result of the comparison principle presented in \Cref{lem:comparison_principle}.  
    The periodicity of $V$ follows from the periodicity of the coefficients and the uniqueness of the solution.
    To show \eqref{eq:appendixB_regularity_1}, by the first statement in \Cref{lem:Mou_Zhang} and the periodicity of functions, it suffices to show that $\| V \|_{L^{\infty}(\mathbb{R}^{d})} \leq C \| r \|_{L^{\infty}(\mathbb{R}^{d})}$ for some constant $C > 0$.
    This follows directly from the comparison principle.
    Specifically, take $M_{\pm} = \pm \| r \|_{L^{\infty}(\mathbb{R}^{d})} / \beta$, then it is straightforward to verify that $M_{+}$ is a supersolution to \eqref{eq:PDE_value_function_2} and $M_{-}$ is a subsolution to \eqref{eq:PDE_value_function_2}.
    By applying \Cref{lem:comparison_principle}, we conclude that $M_{-} \leq V(x) \leq M_{+}$, establishing the $L^{\infty}$ bound of $V$.
    Now by similar arguments in \cite[Theorem 5.2]{mou2021regularity}, we know that the viscosity solution $V$ is also a classical solution.
    Then \eqref{eq:appendixB_regularity_2} follows from \eqref{eq:appendixB_regularity_1} and the second statement in \Cref{lem:Mou_Zhang}. 
 \end{proof}

\begin{lemma}
\label{lem:fractional_Laplace_bound}
    Let $\gamma \in (0, 1)$ and $\alpha \in (0, 1)$.
    For some $C > 0$, we have
    \begin{equation*}
        \| (- \Delta)^{\alpha} f \|_{C^{0, \gamma}(\mathbb{R}^{d})}
        \leq C \| f \|_{C^{2, \gamma}(\mathbb{R}^{d})}.
    \end{equation*}
\end{lemma}
\begin{proof}
    By the definition of the fractional Laplacian, one has
    \begin{equation*}
        \begin{aligned}
            &\quad \, (- \Delta)^{\alpha} f(x)
            = C_{d, \alpha} \, \text{P.V.} \int_{\mathbb{R}^{d}} \frac{f(x) - f(y)}{|x - y|^{d + 2 \alpha}} \, d y\\
            &= C_{d, \alpha} \left ( \text{P.V.} \int_{|x - y| < 1} \frac{f(x) - f(y)}{|x - y|^{d + 2 \alpha}} \, d y
            + \int_{|x - y| > 1} \frac{f(x) - f(y)}{|x - y|^{d + 2 \alpha}} \, d y \right )\\
            &= C_{d, \alpha} \left ( \frac{1}{2} \underbrace{\int_{|z| < 1} \frac{2 f(x) - f(x + z) - f(x - z)}{|z|^{d + 2 \alpha}} \, d z}_{F_{1}(x)}
            + \underbrace{\int_{|z| > 1} \frac{f(x) - f(x + z)}{|z|^{d + 2 \alpha}} \, d z}_{F_{2}(x)} \right ).
        \end{aligned}
    \end{equation*}
    Observe that
    \begin{equation*}
        \| F_{2} \|_{C^{0, \gamma}(\mathbb{R}^{d})}
        \leq 2 \| f \|_{C^{0, \gamma}(\mathbb{R}^{d})} \int_{|z| > 1} \frac{1}{|z|^{d + 2 \alpha}} \, d z
        \leq C \| f \|_{C^{0, \gamma}(\mathbb{R}^{d})}.
    \end{equation*}
    To estimate $F_{1}$, we denote $g(x, z) = (2 f(x) - f(x + z) - f(x - z)) / |z|^{2}$.
    Applying Taylor's theorem, we obtain
    \begin{equation*}
        g(x, z)
        = - \int_{0}^{1} (1 - t) \frac{z \otimes z}{|z|^{2}} : \left [ \nabla^{2} f (x + t z)
        + \nabla^{2} f (x - t z) \right ] \, d t. 
    \end{equation*}
    Given that $f \in C^{2, \gamma}(\mathbb{R}^{d})$, we deduce
    \begin{equation*}
        \| g(\cdot, z) \|_{C^{0, \gamma}(\mathbb{R}^{d})}
        \leq \tilde{C} \| f \|_{C^{2, \gamma}(\mathbb{R}^{d})},
    \end{equation*}
    where $\tilde{C} > 0$ is independent of $z \in \mathbb{R}^{d}$.
    Therefore we have 
    \begin{equation*}
        \| F_{1} \|_{C^{0, \gamma}(\mathbb{R}^{d})}
        \leq \tilde{C} \| f \|_{C^{2, \gamma}(\mathbb{R}^{d})} \int_{|z| < 1} \frac{1}{|z|^{d - 2 + 2 \alpha}} \, d z
        \leq C \| f \|_{C^{2, \gamma}(\mathbb{R}^{d})}.
    \end{equation*}
\end{proof}

\begin{proof}[Proof of \Cref{thm:error_bound}]
    \Cref{eq:PDE_value_function} is reformulated as follows: 
    \begin{equation*}
        \begin{aligned}
            \mathcal{L} V(x)
            := \beta V(x)
            - \drift(x) \cdot \nabla V (x)
            - \diffusionOrdinary(x) : \nabla^{2} V (x)
            + \diffusionFractional(x) (- \Delta)^{\alpha} V (x)
            = r(x).
        \end{aligned}
    \end{equation*}
    Similarly we denote $\widehat{\mathcal{L}}$ as the elliptic operator with coefficients $\widehat{\drift}$, $\widehat{\diffusionOrdinary}$, and $\widehat{\diffusionFractional}$.
    By \Cref{lem:regularity}, we have
    \begin{equation*}
        \| V \|_{C^{2, \gamma}(\mathbb{R}^{d})}
        \leq C \| r \|_{C^{0, \gamma}(\mathbb{R}^{d})}
        \quad \text{and} \quad
        \| \widehat{V} \|_{C^{2, \gamma}(\mathbb{R}^{d})}
        \leq C \| r \|_{C^{0, \gamma}(\mathbb{R}^{d})},
    \end{equation*}
    for some constants $C > 0$ and $\gamma \in (0, 1)$.
    Notice that 
    \begin{equation*}
        \left ( \mathcal{L} - \widehat{\mathcal{L}} \right ) V
        = \left ( \widehat{\drift} - \drift \right ) \cdot \nabla V
        + \left ( \widehat{\diffusionOrdinary} - \diffusionOrdinary \right ) : \nabla^{2} V
        + \left ( \diffusionFractional - \widehat{\diffusionFractional} \right ) (- \Delta)^{\alpha} V.
    \end{equation*}
    Using the fact that $V\in C^{2, \gamma}(\mathbb{R}^{d})$ and the bounds $\| \drift - \widehat{\drift} \|_{C^{0, \gamma}(\mathbb{R}^{d})}, \| \diffusionOrdinary - \widehat{\diffusionOrdinary} \|_{C^{0, \gamma}(\mathbb{R}^{d})}, \| \diffusionFractional - \widehat{\diffusionFractional} \|_{C^{0, \gamma}(\mathbb{R}^{d})} \leq \epsilon$, along with
    \Cref{lem:fractional_Laplace_bound}, we can deduce that
    \begin{equation*}
        \left \| \left ( \mathcal{L} - \widehat{\mathcal{L}} \right ) V \right \|_{C^{0, \gamma}(\mathbb{R}^{d})}
        \leq C \epsilon \| V \|_{C^{2, \gamma}(\mathbb{R}^{d})}.
    \end{equation*}
    for some positive constant $C$. 
    Starting from
    \begin{equation*}
        \widehat{\mathcal{L}} \left ( \widehat{V} - V \right )
        = \left ( \mathcal{L} - \widehat{\mathcal{L}} \right ) V,
    \end{equation*}
    we derive the following bound
    \begin{equation*}
        \left \| \widehat{V} - V \right \|_{C^{2, \gamma}(\mathbb{R}^{d})}
        \leq C \left \| \left ( \mathcal{L} - \widehat{\mathcal{L}} \right ) V \right \|_{C^{0, \gamma}(\mathbb{R}^{d})}
        \leq C \epsilon \| V \|_{C^{2, \gamma}(\mathbb{R}^{d})}
        \leq C \epsilon \| r \|_{C^{0, \gamma}(\mathbb{R}^{d})}
        \leq C \epsilon
    \end{equation*}
    where $C > 0$ is used as a generic constant throughout. 
\end{proof}

\section{Numerical Techniques for the Solution of the Fractional Fokker--Planck Equation and Its Derivatives}\label{appendix:numerical_method}
A detailed description for general dimensions is available in \cite{ye2026fast}.
Here, we summarize the essential components and describe the modifications needed to compute the required gradients.

\subsection{Fractional Fokker--Planck Equation}
Applying the Fourier transform gives
\begin{equation}\label{eq:FFPE_Fourier}
    \left \{
    \begin{aligned}
        \frac{\partial}{\partial t} \widehat{p}(\xi, t)
        &= - \drift(x_{0}) i \xi \widehat{p}(\xi, t)
        - \diffusionOrdinary(x_{0}) |\xi|^{2} \widehat{p}(\xi, t)
        - \diffusionFractional(x_{0}) |\xi|^{2 \alpha} \widehat{p}(\xi, t)\\
        \widehat{p}(\xi, 0)
        &= \exp \left ( - i \xi x_{0} \right )
    \end{aligned}
    \right ..
\end{equation}
The solution of \Cref{eq:FFPE_Fourier} is
\begin{equation}
    \begin{aligned}
        \widehat{p}(\xi, t)
        &= \exp \left ( - i \xi \big ( x_{0} + \drift(x_{0}) t \big ) \right )
        \exp \left ( - \left ( \diffusionOrdinary(x_{0}) |\xi|^{2} + \diffusionFractional(x_{0}) |\xi|^{2 \alpha} \right ) t \right ).
    \end{aligned}
\end{equation}
Inverse Fourier transform gives
\begin{equation}\label{eq:solution_representation_raw}
    \begin{aligned}
        &\quad \ 2 \pi p(x, t)
        = \int_{- \infty}^{\infty} \widehat{p}(\xi, t) \exp( i \xi x ) \, d \xi\\
        &= \int_{- \infty}^{\infty} \underbrace{\exp \left ( i \xi \big ( x - x_{0} - \drift(x_{0}) t \big ) \right )}_{\text{oscillating part}}
        \underbrace{\exp \left ( - \left ( \diffusionOrdinary(x_{0}) |\xi|^{2} + \diffusionFractional(x_{0}) |\xi|^{2 \alpha} \right ) t \right )}_{\text{decay part}} \, d \xi.
    \end{aligned}
\end{equation}
We can further simplify \Cref{eq:solution_representation_raw} into
\begin{equation}\label{eq:solution_representation}
    \begin{aligned}
        &\quad \, p \left ( x, t; x_{0}, \alpha, \drift(x_{0}), \diffusionOrdinary(x_{0}), \diffusionFractional(x_{0}) \right )\\
        &= \frac{1}{\pi} \int_{0}^{\infty} \cos \left ( \xi \big ( x - x_{0} - \drift(x_{0}) t \big ) \right )
        \left ( - \left ( \diffusionOrdinary(x_{0}) \xi^{2} + \diffusionFractional(x_{0}) \xi^{2 \alpha} \right ) t \right ) \, d \xi.
    \end{aligned}
\end{equation}
Then apply the numerical method in \cite{ye2026fast} to compute the solution efficiently.

\subsection{Gradients Computation}
Based on \Cref{eq:log_likelihood_function_approximation}, to compute the approximated gradients, we only need to derive $\nabla \ln p \left ( x_{(j + 1) \Delta t}^{(i)}, \Delta t; x_{j \Delta t}^{(i)}, \alpha, \Theta(x_{j \Delta t}^{(i)}; \theta) \right )$ and then sum them up. Essentially, we merely need to compute both $p$ and $\nabla p$ for any given data ($x_{(j + 1) \Delta t}^{(i)}, x_{j \Delta t}^{(i)}, \Delta t$).

Using the basis functions to expand the coefficients $\drift(x)$, $\diffusionOrdinary(x)$, $\diffusionFractional(x)$ as in \Cref{eq:basis_function_expansion}, we can rewrite \Cref{eq:solution_representation} as
\begin{equation}\label{eq:solution_representation_with_expansion}
    \begin{aligned}
        &\quad \ \, p \left ( x_{(j + 1) \Delta t}^{(i)}, \Delta t; x_{j \Delta t}^{(i)}, \alpha, \Theta(x_{j \Delta t}^{(i)}; \theta) \right )\\
        &= \frac{1}{\pi} \int_{0}^{\infty} \cos \left ( \underbrace{\xi \left ( x_{(j + 1) \Delta t}^{(i)} - x_{j \Delta t}^{(i)} - \Delta t \sum_{k = 1}^{K} \theta_{1, k} \phi_{1, k}(x_{j \Delta t}^{(i)}) \right )}_{I_{1} := \text{interior1}} \right )\\
        &\qquad \exp \left ( \underbrace{- \left ( \xi^{2} \sum_{k = 1}^{K} \theta_{2, k} \phi_{2, k}(x_{j \Delta t}^{(i)}) + \xi^{2 \alpha} \sum_{k = 1}^{K} \theta_{3, k} \phi_{3, k}(x_{j \Delta t}^{(i)}) \right ) \Delta t}_{I_{2} := \text{interior2}} \right ) \, d \xi.
    \end{aligned}
\end{equation}
We can also compute that
\begin{equation}\label{eq:solution_representation_derivative}
    \left \{
    \begin{aligned}
        \frac{\partial p}{\partial \theta_{1, k}}
        &= \frac{\Delta t}{\pi} \phi_{1, k}(x_{j \Delta t}^{(i)}) \int_{0}^{\infty} \xi \sin(I_{1}) \exp(I_{2}) \, d \xi,\\
        \frac{\partial p}{\partial \theta_{2, k}}
        &= - \frac{\Delta t}{\pi} \phi_{2, k}(x_{j \Delta t}^{(i)}) \int_{0}^{\infty} \xi^{2} \cos(I_{1}) \exp(I_{2}) \, d \xi,\\
        \frac{\partial p}{\partial \theta_{3, k}}
        &= - \frac{\Delta t}{\pi} \phi_{3, k}(x_{j \Delta t}^{(i)}) \int_{0}^{\infty} \xi^{2 \alpha} \cos(I_{1}) \exp(I_{2}) \, d \xi.
    \end{aligned}
    \right .
\end{equation}

Essentially, we have
\begin{equation}\label{eq:solution_representation_derivative_lnp}
    \left \{
    \begin{aligned}
        \frac{\partial \ln p}{\partial \theta_{1, k}}
        &= \Delta t \phi_{1, k}(x_{j \Delta t}^{(i)}) \frac{\int_{0}^{\infty} \xi \sin(I_{1}) \exp(I_{2}) \, d \xi}{\int_{0}^{\infty} \cos(I_{1}) \exp(I_{2}) \, d \xi},\\
        \frac{\partial \ln p}{\partial \theta_{2, k}}
        &= - \Delta t \phi_{2, k}(x_{j \Delta t}^{(i)}) \frac{\int_{0}^{\infty} \xi^{2} \cos(I_{1}) \exp(I_{2}) \, d \xi}{\int_{0}^{\infty} \cos(I_{1}) \exp(I_{2}) \, d \xi},\\
        \frac{\partial \ln p}{\partial \theta_{3, k}}
        &= - \Delta t \phi_{3, k}(x_{j \Delta t}^{(i)}) \frac{\int_{0}^{\infty} \xi^{2 \alpha} \cos(I_{1}) \exp(I_{2}) \, d \xi}{\int_{0}^{\infty} \cos(I_{1}) \exp(I_{2}) \, d \xi}.
    \end{aligned}
    \right .
\end{equation}

\begin{remark}
    If we further expand $\displaystyle \alpha(x) = \sum_{k = 1}^{K} \theta_{4, k} \phi_{4, k}(x)$, then the corresponding derivatives are
    \begin{equation*}
        \left \{
        \begin{aligned}
            \frac{\partial p}{\partial \theta_{4, k}}
            &= - \frac{2 \Delta t}{\pi} \phi_{4, k}(x_{j \Delta t}^{(i)}) \left [ \sum_{k = 1}^{K} \theta_{3, k} \phi_{3, k}(x_{j \Delta t}^{(i)}) \right ] \int_{0}^{\infty} \xi^{2 \alpha} \ln(\xi) \cos(I_{1}) \exp(I_{2}) \, d \xi,\\
            \frac{\partial \ln p}{\partial \theta_{4, k}}
            &= - 2 \Delta t \phi_{4, k}(x_{j \Delta t}^{(i)}) \left [ \sum_{k = 1}^{K} \theta_{3, k} \phi_{3, k}(x_{j \Delta t}^{(i)}) \right ] \frac{\int_{0}^{\infty} \xi^{2 \alpha} \ln(\xi) \cos(I_{1}) \exp(I_{2}) \, d \xi}{\int_{0}^{\infty} \cos(I_{1}) \exp(I_{2}) \, d \xi}.
        \end{aligned}
        \right .
    \end{equation*}
\end{remark}

\subsection{General Dimension}
In general, the solution is represented by
\begin{equation}
    \begin{aligned}
        &\quad \, \tilde{p}(y, t; \boldsymbol{x}_{0}, \alpha, \boldsymbol{b}, \diffusionOrdinary, \diffusionFractional)
        = p(\boldsymbol{x}, t; \boldsymbol{x}_{0}, \alpha, \boldsymbol{b}, \diffusionOrdinary, \diffusionFractional)\\
        &= \frac{1}{y^{(d - 2) / 2}} \int_{0}^{\infty} \left ( \frac{r}{2 \pi} \right )^{d / 2} J_{(d - 2) / 2}(y r) \exp \left ( - \left ( \diffusionOrdinary r^{2} + \diffusionFractional r^{2 \alpha} \right ) t \right ) \, d r,
    \end{aligned}
\end{equation}
where $y := \left | \boldsymbol{x} - \boldsymbol{x}_{0} - \boldsymbol{b} t \right |$ and $J_{\nu}$ is the Bessel function of the first kind.

Differentiating with respect to $y$, $\diffusionOrdinary$, and $\diffusionFractional$ gives
\begin{equation}
    \begin{aligned}
        \frac{\partial \tilde{p}}{\partial y}(y, t)
        &= \frac{1}{y^{d / 2}} \int_{0}^{\infty} \left ( \frac{r}{2 \pi} \right )^{d / 2} \left ( y r J_{(d - 4) / 2}(y r) - (d - 2) J_{(d - 2) / 2}(y r) \right )\\
        &\qquad \qquad \qquad \qquad \exp \left ( - \left ( \diffusionOrdinary r^{2} + \diffusionFractional r^{2 \alpha} \right ) t \right ) \, d r,\\
        \frac{\partial \tilde{p}}{\partial \diffusionOrdinary}(y, t)
        &= \frac{- t}{y^{(d - 2) / 2}} \int_{0}^{\infty} \left ( \frac{r}{2 \pi} \right )^{d / 2} r^{2} J_{(d - 2) / 2}(y r) \exp \left ( - \left ( \diffusionOrdinary r^{2} + \diffusionFractional r^{2 \alpha} \right ) t \right ) \, d r,\\
        \frac{\partial \tilde{p}}{\partial \diffusionFractional}(y, t)
        &= \frac{- t}{y^{(d - 2) / 2}} \int_{0}^{\infty} \left ( \frac{r}{2 \pi} \right )^{d / 2} r^{2 \alpha} J_{(d - 2) / 2}(y r) \exp \left ( - \left ( \diffusionOrdinary r^{2} + \diffusionFractional r^{2 \alpha} \right ) t \right ) \, d r.
    \end{aligned}
\end{equation}

With a higher-dimensional Fourier basis expansion, the same construction yields \Cref{alg:Adam_with_tail_correction} for the higher-dimensional setting.

\section{Algorithm Details of the Robust Maximum Likelihood Estimation}\label{appendix:maximum_likelihood_estimation}
This section presents the omitted details from \Cref{alg:Adam_with_tail_correction} along with relevant explanations.

\subsection{Algorithms}
\Cref{alg:Dirac_integration_with_singularity,alg:Dirac_integration_for_slow_decay,alg:Dirac_integration_for_slow_decay,alg:Dirac_basic_solver,alg:FFPE_direct_gradient,alg:FFPE_finite_difference_gradient,alg:TCF_computation}, which were omitted from \Cref{alg:Adam_with_tail_correction}, are presented in this subsection.

\begin{algorithm}[htbp]
    \caption{Integration with Singularity}\label{alg:Dirac_integration_with_singularity}
    \begin{algorithmic}[1]
        \Require A smooth function denoted by $f$, a fractional exponent signified by $\alpha \in (0, 1)$, and a pseudo-temporal coefficient represented by $\tau > 0$
        \Ensure Computation of the integral $\displaystyle \int_{0}^{1} f(z) \exp \left ( - z^{2 \alpha} \tau \right ) \, d z$
        
        \If{$\tau > \textit{1}$ \textbf{or} precomputed}
            \State $\{ x_{j} \}_{j = 1}^{16} \gets $ Gaussian quadrature points in $[0, 1]$
            \State $\{ w_{j} \}_{j = 1}^{16} \gets $ solve from $\int_{0}^{1} f(z) \exp \left ( - z^{2 \alpha} \tau \right ) \, d z = \sum_{j = 1}^{16} w_{j} f(x_{j})$ for Legendre polynomials of order $0, \ldots, 15$
            \Statex \Comment{Can be precomputed for repetitive evaluations under identical $\alpha$, and $\tau$}
            \State \Return $\sum_{j = 1}^{N} w_{j} f(x_{j})$
        \Else
            \State Choose $K$ as follow:
            \Statex
            \begin{center}
                \begin{tikzpicture}[scale=1]
                    \draw[->] (0, 0) -- (10, 0) node[right] {$\tau$};
\coordinate (tick1) at (3, 0);
\coordinate (tick2) at (5, 0);
\coordinate (tick3) at (7, 0);
\coordinate (tick4) at (9, 0);

\foreach \tick in {tick1, tick2, tick3, tick4}
{
    \draw (\tick) -- ++(0, +0.2);
}

\node[below] at (tick1) {\textit{1e-3}};
\node[below] at (tick2) {\textit{1e-2}};
\node[below] at (tick3) {\textit{1e-1}};
\node[below] at (tick4) {\textit{1}};

\node[above] at (2, 0) {$K_{4} = 4$};
\node[above] at (4, 0) {$K_{3} = 6$};
\node[above] at (6, 0) {$K_{2} = 9$};
\node[above] at (8, 0) {$K_{1} = 17$};

                \end{tikzpicture}
            \end{center}
            \State \Return $\displaystyle \sum_{\kappa = 0}^{K} \frac{(- \tau)^{\kappa}}{\kappa!} \int_{0}^{1} f(z) z^{2 \alpha \kappa} \, d z$, where the integrals here are evaluated by the Gauss-Jacobi quadrature
        \EndIf
    \end{algorithmic}
\end{algorithm}

\begin{algorithm}[htbp]
    \caption{Integration for Slow Decay}\label{alg:Dirac_integration_for_slow_decay}
    \begin{algorithmic}[1]
        \Require A smooth oscillatory function denoted by $f$, a fractional exponent signified by $\alpha \in (0, 1)$, and a pseudo-temporal coefficient represented by $\tau > 0$
        \Ensure Computation of the integral $\displaystyle \int_{1}^{\infty} f(z) \exp \left ( - z^{2 \alpha} \tau \right ) \, d z$

        \State $M, M_{\text{max}}, \varepsilon \gets 80, 5120, \textit{1e-14}$
            \Comment{Adjustable}
        \State $I_{\text{previous}}, I_{\text{current}} \gets 0, \infty$
        \While{$|I_{\text{current}} - I_{\text{previous}}| > \varepsilon$ \textbf{and} $M \leq M_{\text{max}}$}
            \State $I_{\text{previous}} \gets I_{\text{current}}$
            \State $I_{\text{current}} \gets$ Apply quadrature to $\displaystyle \int_{1}^{M} f(z) \exp(- z^{2 \alpha} \tau) w_{M}(z) \, d z$
            \Statex \Comment{Utilize precomputed quadrature points and weights, scaling the number of quadrature points with $M$}
            \State $M \gets 2 M$
        \EndWhile
        \If{$|I_{\text{current}} - I_{\text{previous}}| > \varepsilon$}
            \State \textbf{raise} a flag (without stopping)
        \EndIf
        \State \Return $I_{\text{current}}$
    \end{algorithmic}
\end{algorithm}

The windowing function $w_{M}$ in \Cref{alg:Dirac_integration_for_slow_decay} is defined as
\begin{equation*}
    w_{M}(z)
    = \left \{
    \begin{aligned}
        &1, && s \leq 0\\
        &\exp \left ( - 2 \frac{\exp(- 1 / s^{2})}{(1 - s)^{2}} \right ), && 0 < s < 1\\
        &0, && s \geq 1
    \end{aligned}
    \right .,
\end{equation*}
where the parameter $s$ is given by
\begin{equation*}
    s(z)
    = \frac{2 |z|}{M} - 1,
\end{equation*}
with $M > 0$.

\begin{algorithm}[htbp]
    \caption{Fundamental Algorithm}\label{alg:Dirac_basic_solver}
    \begin{algorithmic}[1]
        \Require A parameterized function denoted by $f_{(y, \diffusionOrdinary, t)}$, magnitude of the displacement $y$, temporal parameter $t > 0$, coefficients $\diffusionOrdinary \geq 0$, $\diffusionFractional > 0$, and a fractional exponent $\alpha \in (0, 1)$
        \Ensure Computation of the integral $\displaystyle p = \int_{0}^{\infty} f_{(y, \diffusionOrdinary, t)}(r) \exp \left ( - r^{2 \alpha} \diffusionFractional t \right ) \, d r$

        \If{$y \leq 10$}
            \State $p \gets $ \Call{\Cref{alg:Dirac_integration_for_slow_decay}}{$f_{(y, \diffusionOrdinary, t)}$, $\alpha$, $\diffusionFractional t$} $+$ \Call{\Cref{alg:Dirac_integration_with_singularity}}{$f_{(y, \diffusionOrdinary, t)}$, $\alpha$, $\diffusionFractional t$}
        \EndIf
            \Comment{Apply force scaling when $y$ is large}
        \If{$y > 10$ \textbf{or} $|I_{\text{current}} - I_{\text{previous}}| > \varepsilon$ remains true in \Call{\Cref{alg:Dirac_integration_for_slow_decay}}{}}
            \State $h \gets $ the exponent of $r$ in $f_{(y, \diffusionOrdinary, t)}(r)$
                \Comment{e.g. $h = 2 \alpha$ for $r^{2 \alpha} \cos \left ( y r \right ) \exp \left ( - \diffusionOrdinary r^{2} t \right )$}
            \State $y^{\text{scaled}} \gets \pi / 2$;
            $t^{\text{scaled}} \gets \left ( \frac{y^{\text{scaled}}}{y} \right )^{2 \alpha} t$;
            $\diffusionOrdinary^{\text{scaled}} \gets \left ( \frac{y^{\text{scaled}}}{y} \right )^{2 - 2 \alpha} \diffusionOrdinary$
                \Comment{The scaling law}
            \State $p \gets$ $\left ( \frac{y^{\text{scaled}}}{y} \right )^{h + 1}$ $\times$ [ \Call{\Cref{alg:Dirac_integration_for_slow_decay}}{$f_{(y, \diffusionOrdinary, t)^{\text{scaled}}}$, $\alpha$, $\diffusionFractional t^{\text{scaled}}$} $+$ \Call{\Cref{alg:Dirac_integration_with_singularity}}{$f_{(y, \diffusionOrdinary, t)^{\text{scaled}}}$, $\alpha$, $\diffusionFractional t^{\text{scaled}}$} ]
        \EndIf
        \State \Return $p$
    \end{algorithmic}
\end{algorithm}

\begin{algorithm}[htbp]
    \caption{Direct Gradient Computation}\label{alg:FFPE_direct_gradient}
    \begin{algorithmic}[1]
        \Require Parameters $\theta = \{ \theta_{l, k} \}_{l = 1, k = 1}^{l = 3, k = K}$, current state $x_{\textnormal{current}}$, next state $x_{\textnormal{next}}$, temporal difference $\Delta t > 0$ and a fractional exponent $\alpha \in (0, 1)$
        \Ensure The direct computation of the gradient of $- \ell(\theta)$ using one sample
        \State $y \gets x_{\textnormal{next}} - x_{\textnormal{current}} - \Delta t \sum_{k = 1}^{K} \theta_{1, k} \phi_{1, k}(x_{\textnormal{current}})$
        \State $\diffusionOrdinary, \diffusionFractional \gets \max(0, \sum_{k = 1}^{K} \theta_{2, k} \phi_{2, k}(x_{\textnormal{current}})), \max(\textit{1e-8}, \sum_{k = 1}^{K} \theta_{3, k} \phi_{3, k}(x_{\textnormal{current}}))$
        \State $f \gets \text{Function } y, \diffusionOrdinary, t \mapsto \left [ \text{Function } r \mapsto \cos \left ( y r \right ) \exp \left ( - \diffusionOrdinary r^{2} t \right ) \right ]$
        \State $f_{\drift} \gets \text{Function } y, \diffusionOrdinary, t \mapsto \left [ \text{Function } r \mapsto r \sin \left ( y r \right ) \exp \left ( - \diffusionOrdinary r^{2} t \right ) \right ]$
        \State $f_{\diffusionOrdinary} \gets \text{Function } y, \diffusionOrdinary, t \mapsto \left [ \text{Function } r \mapsto r^{2} \cos \left ( y r \right ) \exp \left ( - \diffusionOrdinary r^{2} t \right ) \right ]$
        \State $f_{\diffusionFractional} \gets \text{Function } y, \diffusionOrdinary, t \mapsto \left [ \text{Function } r \mapsto r^{2 \alpha} \cos \left ( y r \right ) \exp \left ( - \diffusionOrdinary r^{2} t \right ) \right ]$

        \State $p \gets $ \Call{\Cref{alg:Dirac_basic_solver}}{$f$, $y$, $\Delta t$, $\diffusionOrdinary$, $\diffusionFractional$, $\alpha$}
        \State $p_{\drift} \gets $ $\Delta t$ $\times$ \Call{\Cref{alg:Dirac_basic_solver}}{$f_{\drift}$, $y$, $\Delta t$, $\diffusionOrdinary$, $\diffusionFractional$, $\alpha$}
        \State $p_{\diffusionOrdinary} \gets $ $- \Delta t$ $\times$ \Call{\Cref{alg:Dirac_basic_solver}}{$f_{\diffusionOrdinary}$, $y$, $\Delta t$, $\diffusionOrdinary$, $\diffusionFractional$, $\alpha$}
        \State $p_{\diffusionFractional} \gets $ $- \Delta t$ $\times$ \Call{\Cref{alg:Dirac_basic_solver}}{$f_{\diffusionFractional}$, $y$, $\Delta t$, $\diffusionOrdinary$, $\diffusionFractional$, $\alpha$}
        \State \Return
        \Statex \hspace{\algorithmicindent}
        $\begin{aligned}[t]
            - \nabla_{\theta} \ell(\theta)
            =
            - \big[
                &[ \phi_{1, k}(x_{\textnormal{current}}) p_{\drift} / p ]_{k = 1}^{K},\\
                &[ \phi_{2, k}(x_{\textnormal{current}}) p_{\diffusionOrdinary} / p ]_{k = 1}^{K},\\
                &[ \phi_{3, k}(x_{\textnormal{current}}) p_{\diffusionFractional} / p ]_{k = 1}^{K}
            \big].
        \end{aligned}$
    \end{algorithmic}
\end{algorithm}

\begin{algorithm}[htbp]
    \caption{Finite Difference Gradient Computation}\label{alg:FFPE_finite_difference_gradient}
    \begin{algorithmic}[1]
        \Require Parameters $\theta = \{ \theta_{l, k} \}_{l = 1, k = 1}^{l = 3, k = K}$, current position $x_{\textnormal{current}}$, next position $x_{\textnormal{next}}$, temporal difference $\Delta t > 0$ and a fractional exponent $\alpha \in (0, 1)$
        \Ensure The finite difference approximation of the gradient of $- \ell(\theta)$ using one sample
        \State $y \gets x_{\textnormal{next}} - x_{\textnormal{current}} - \Delta t \sum_{k = 1}^{K} \theta_{1, k} \phi_{1, k}(x_{\textnormal{current}})$
        \State $\diffusionOrdinary, \diffusionFractional \gets \max(0, \sum_{k = 1}^{K} \theta_{2, k} \phi_{2, k}(x_{\textnormal{current}})), \max(\textit{1e-8}, \sum_{k = 1}^{K} \theta_{3, k} \phi_{3, k}(x_{\textnormal{current}}))$
        \State $f \gets \text{Function } y, \diffusionOrdinary, t \mapsto \left [ \text{Function } r \mapsto \cos \left ( y r \right ) \exp \left ( - \diffusionOrdinary r^{2} t \right ) \right ]$

        \State $p \gets $ \Call{\Cref{alg:Dirac_basic_solver}}{$f$, $y$, $\Delta t$, $\diffusionOrdinary$, $\diffusionFractional$, $\alpha$}
        \State $\nabla_{\theta} \ell(\theta) \gets \boldsymbol{0}$; randomly select at most $B^{\text{FD}} = 10$ components of $\theta$
        \Comment{$B^{\text{FD}}$ is adjustable}
        \For{each selected component with index $m$}
            \State $\Delta \theta \sim \pm \text{Uniform}(0.001, 0.1)$
            \State $\tilde{\theta} \gets \theta + \boldsymbol{e}_{m} \Delta \theta$
                \Comment{$\boldsymbol{e}_{m}$ is the unit vector with a $1$ in the $m$-th position}
            \State $y^{\text{FD}}, \diffusionOrdinary^{\text{FD}}, \diffusionFractional^{\text{FD}} \gets $ calculated from $\tilde{\theta}$
            \State $p^{\text{FD}} \gets $ \Call{\Cref{alg:Dirac_basic_solver}}{$f$, $y^{\text{FD}}$, $\Delta t$, $\diffusionOrdinary^{\text{FD}}$, $\diffusionFractional^{\text{FD}}$, $\alpha$}
            \State $\nabla_{\theta} \ell(\theta) \gets \nabla_{\theta} \ell(\theta) + \boldsymbol{e}_{m} (\ln(p^{\text{FD}}) - \ln(p)) / \Delta \theta$
        \EndFor
        \State \Return $- \nabla_{\theta} \ell(\theta)$
    \end{algorithmic}
\end{algorithm}

The variables $y$, $\diffusionOrdinary$, and $\diffusionFractional$ in \Cref{alg:FFPE_direct_gradient} or \Cref{alg:FFPE_finite_difference_gradient} will vary depending on the specific values of $x_{j \Delta t}$ and $x_{(j + 1) \Delta t}$.
To compute a batch of evaluations, one can iterate over all the samples in the batch, compute the gradient for each individual sample, and then sum these gradients or take their average.

\begin{algorithm}[htbp]
    \caption{Tail Correction Factor Computation}\label{alg:TCF_computation}
    \begin{algorithmic}[1]
        \Require Parameters $\theta = \{ \theta_{l, k} \}_{l = 1, k = 1}^{l = 3, k = K}$, temporal difference $\Delta t > 0$, the cutting threshold \texttt{CT}, position difference median $\mu$, tail part ratio in sample $\texttt{R}_{\text{sample}}$ and a fractional exponent $\alpha \in (0, 1)$
        \Ensure The tail correction factor \texttt{TCF}

        \State $\drift, \diffusionOrdinary, \diffusionFractional \gets $ means of $\sum_{k = 1}^{K} \theta_{1, k} \phi_{1, k}(x), \sum_{k = 1}^{K} \theta_{2, k} \phi_{2, k}(x), \sum_{k = 1}^{K} \theta_{3, k} \phi_{3, k}(x)$
        \Statex \Comment{Could be different kinds of mean (arithmetic mean gives $\theta_{1, 1}, \theta_{2, 1}, \theta_{3, 1}$ for Fourier basis)}
        \State $f \gets \text{Function } y, \diffusionOrdinary, t \mapsto \left [ \text{Function } r \mapsto \cos \left ( y r \right ) \exp \left ( - \diffusionOrdinary r^{2} t \right ) \right ]$

        \State $p \gets \text{Function } x \mapsto$ \Call{\Cref{alg:Dirac_basic_solver}}{$f$, $|x - \drift \Delta t|$, $\Delta t$, $\diffusionOrdinary$, $\diffusionFractional$, $\alpha$}
        \State $\texttt{R}_{\theta} \gets \int_{|x - \mu| > \texttt{CT}} p(x) \, d x$
            \Comment{Numerical integration}
        \State \Return $\texttt{TCF} = \max(0, ( \texttt{R}_{\theta} - \texttt{R}_{\text{sample}} ) / ( 1 - \texttt{R}_{\text{sample}} ))$
            \Comment{In case $\texttt{R}_{\theta} < \texttt{R}_{\text{sample}}$}
    \end{algorithmic}
\end{algorithm}

\subsection{Moving Average}
The necessity of using a moving average in \Cref{alg:Adam_with_tail_correction} arises from the instability of the parameters associated with the coefficient $\diffusionFractional$.
These parameters are significantly affected by the heavy-tailed nature of the data.
When the tail part sample pool is visited, the parameters related to $\diffusionFractional$ exhibit considerable fluctuations.
Consequently, it is not appropriate to use the raw parameters to assess convergence.
Instead, employing a moving average is a more reasonable approach.
In practice, we compute the moving average over 20,000 steps to enhance the stability and reliability of our method.
See \Cref{fig:moving_average_example} for an example taken from an estimation example in \Cref{eg:censored_MCMC_constant} with a total of 84,000 steps.
The oscillations for $\diffusionFractional$ are particularly large in \Cref{fig:moving_average_example}, demonstrating the necessity of using the moving average to improve stability and accuracy in learning, as the raw learning history exhibits significant volatility.
\begin{figure}[htp]
    \centering
    \includegraphics[width=\linewidth]{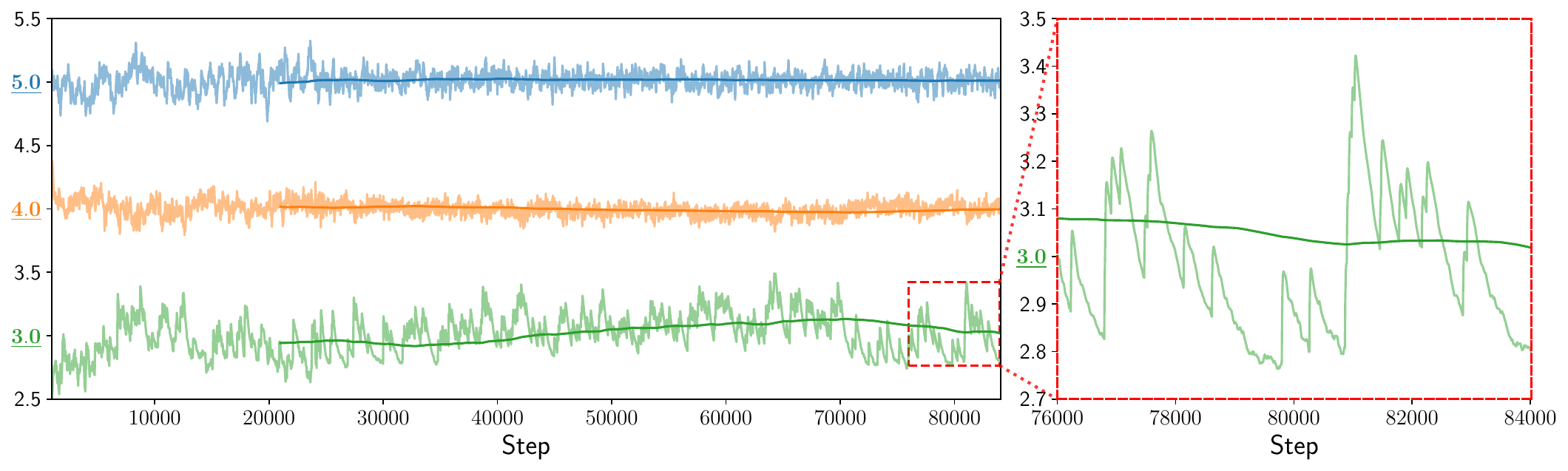}
    \caption{
        The left-hand graph illustrates the learning trajectories of three coefficients.
        The blue curve represents $\drift$ with a ground truth value of $5$, the orange curve represents $\diffusionOrdinary$ with a ground truth value of $4$, and the green curve represents $\diffusionFractional$ with a ground truth value of $3$.
        The lighter oscillatory lines show the actual learning histories, while the bold lines represent the moving averages of 20,000 steps.
        The right-hand graph is an enlarged view of the section highlighted by the red rectangle on the left, focusing on the coefficient $\diffusionFractional$.
    }
    \label{fig:moving_average_example}
\end{figure}

If \texttt{TCF} is too small, achieving statistical significance necessitates a greater number of iterations and an increased moving average window size, approximately on the order of $\Omega(n_{\textnormal{min}} / \texttt{TCF})$, where $n_{\textnormal{min}}$ denotes the minimum sample size required for statistical significance.

\subsection{Nonuniform Temporal Sampling}
Moreover, the trajectories in the observation data do not need to be of the same length.
The uniform length notation used in \Cref{eq:main_pool} and \Cref{alg:Adam_with_tail_correction} is merely for clarity.
\Cref{alg:Adam_with_tail_correction} can be further extended to accommodate uneven temporal sampling by replacing $\texttt{CT}$ with $\texttt{CT}(\Delta t)$, which depends on the time difference $\Delta t = t_{\textnormal{next}} - t_{\textnormal{current}}$.
First, compute the drift mean as follows:
\begin{equation*}
    \mu \gets \textnormal{median} \left ( \Delta x / \Delta t \Big | \left ( x_{\textnormal{current}}, x_{\textnormal{next}}, t_{\textnormal{current}}, t_{\textnormal{next}} \right ) \in P_{\textnormal{main}} \right ).
\end{equation*}
Then, the tail part sample pool can be obtained by:
\begin{equation*}
    P_{\textnormal{tail}} \gets \left \{ \left ( x_{\textnormal{current}}, x_{\textnormal{next}}, t_{\textnormal{current}}, t_{\textnormal{next}} \right ) \in P_{\textnormal{main}} \bigg | \left | \Delta x - \mu \Delta t \right | > \texttt{CT} \left ( \Delta t \right ) \right \}.
\end{equation*}

\end{document}